\documentclass[a4paper]{article}

\usepackage{amsmath,amssymb,mathtools,amsthm}
\usepackage{xcolor,graphicx,float,enumerate}

\usepackage[font=small]{caption}
\usepackage{subcaption}

\usepackage[hidelinks]{hyperref}
\usepackage[noabbrev,nameinlink]{cleveref}

\usepackage[utf8]{inputenc}
\DeclareUnicodeCharacter{00A0}{ }

\usepackage[a4paper]{geometry}

\usepackage{esint}

\newtheorem{proposition}{Proposition}[section]
\newtheorem{theorem}[proposition]{Theorem}

\newtheorem{lemma}[proposition]{Lemma}
\theoremstyle{definition}\newtheorem{definition}[proposition]{Definition}

\newtheorem{remark}[proposition]{Remark}
\newtheorem*{theorem*}{Theorem}

\newtheorem{hypothesis}{Hypothesis}

\numberwithin{equation}{section}

\newcommand{\email}[1]{\href{mailto:#1}{#1}}

\allowdisplaybreaks[1]

\usepackage{amsmath,amssymb}

\newcommand{\eps}{\varepsilon}

\newcommand{\G}{\mathbb{G}}

\newcommand{\M}{\mathbb{M}}
\newcommand{\K}{\mathbb{K}}

\newcommand{\p}{\mathbb S}

\newcommand{\Ls}{\ensuremath{\mathrm{L}}}
\newcommand{\E}{\ensuremath{\mathrm{E}}}
\newcommand{\D}{\ensuremath{\mathrm{D}}}
\newcommand{\J}{\ensuremath{\mathrm{J}}}

\newcommand{\Green}{\mathcal{G}}
\newcommand{\Heat}{\mathcal{H}}

\newcommand{\Semi}{\mathcal{S}}

\DeclareMathOperator{\Laplace}{\mathcal L}
\DeclareMathOperator{\RLInt}{\mathcal I}

\newcommand{\re}{\mathbb{R}}
\newcommand{\R}{\re}
\newcommand{\N}{\mathbb{N}}

\newcommand{\bG}{\mathbb{G}}
\newcommand{\bP}{\mathbb{P}}
\newcommand{\bS}{\mathbb{S}}

\newcommand{\cH}{{\mathcal H}}
\newcommand{\cP}{{\mathcal P}}
\newcommand{\cS}{{\mathcal S}}

\DeclareMathOperator{\sign}{sign}

\newcommand{\angles}[1]{\left\langle#1\right\rangle}

\newcommand{\ceil}[1]{\left \lceil #1 \right \rceil}

\newcommand{\norm}[2][]{\left\|#2\right\|_{#1}}

\DeclareMathOperator{\dist}{dist}

\newcommand{\caputo}[2][]{\prescript{C}{#1}{\partial}_{#2}^\alpha}
\newcommand{\riemann}[2][]{\prescript{R}{#1}{\partial}_{#2}^\alpha}
\newcommand{\riemannOrder}[3][]{\prescript{R}{#1}{\partial}_{#2}^{#3}}

\newcommand{\dtfrac}[2][]{\prescript{\bullet}{#1}{\partial}_{#2}^\alpha}

\newcommand*\diff{\mathop{}\!\mathrm{d}}
\newcommand{\defeq}{\vcentcolon=}

\usepackage[
url=false,
isbn=false,
backend=biber,
maxcitenames=4,
giveninits=true,
uniquename=init,
maxbibnames=100,
doi=false,
block=none]{biblatex}

\renewbibmacro{in:}{}
\DeclareFieldFormat[article]{citetitle}{#1}
\DeclareFieldFormat[article]{title}{#1}
\title{Singular solutions for space-time \\fractional equations in a bounded domain}
\author{%
	Hardy Chan\thanks{Department of Mathematics and Computer Science, University of Basel.
	\email{hardy.chan@unibas.ch}}
    \and
    David G\'omez-Castro\thanks{Instituto de Matemática Interdisciplinar, Universidad Complutense de Madrid.
    	\email{dgcastro@ucm.es}}
    \and	
    Juan Luis V\'{a}zquez\thanks{Depto. de Matem\'aticas
    	Univ. Aut\'onoma de Madrid.
    	\email{juanluis.vazquez@uam.es}}
}

\begin{document}
	\maketitle

\begin{abstract}
This paper is devoted to describing a linear diffusion problem involving fractional-in-time derivatives and self-adjoint integro-differential space operators posed in bounded domains. One main concern of our paper is to deal with singular boundary data which are typical of fractional diffusion operators in space, and the other one is the consideration of the fractional-in-time Caputo and Riemann--Liouville derivatives in a unified way.
We first construct classical solutions of our problems using the spectral theory and discussing the corresponding fractional-in-time ordinary differential equations. 
We take advantage of the duality between these fractional-in-time derivatives to introduce the notion of weak-dual solution for weighted-integrable data.
As the main result of the paper, we prove the well-posedness of the initial and boundary-value problems in this sense. 
\end{abstract}

\section{Introduction}

This paper is devoted to describing a linear diffusion problem involving fractional-in-time derivatives and self-adjoint integro-differential space operators in bounded domains. 
More precisely, we want to extend the results in \cite{Chan+GC+Vazquez2020parabolic} to the fractional-in-time setting. Let $\Omega$ be a bounded smooth domain of $\mathbb R^d$ and $\Ls$ be an elliptic operator of order $2s\in(0,2)$ as in \cite{Chan+GC+Vazquez2020parabolic}. The precise hypotheses of $\Ls$ are made on its Green and heat kernels in \Cref{sec:main results}. Consider
\begin{gather}
	\tag{P$_\bullet$}
	\label{eq:main}
	\begin{dcases}
		\prescript{\bullet}{}\partial_t^{\alpha}u(t,x)
		+\Ls u(t,x)=f(t,x),
		& x\in\Omega, t\in(0,T),\\
		u(t,x)=0,
		& x\in\Omega^{c}, t>0,\\
		\lim_{x \to \zeta} \frac{  u(t,x) }{u^\star(x)} =h(t,\zeta),
		& \zeta\in\partial\Omega, t>0,
	\end{dcases}
\end{gather}
where $u^\star$ is a canonically-chosen representative of a class of solutions of $\Ls u = 0$ which are singular on the boundary, which we will explain below after the statement of \Cref{thm:h = 0}, and $\prescript{\bullet}{}\partial_t^\alpha$ ($0< \alpha < 1$) is either the Caputo derivative $\caputo{t}$ or the Riemann--Liouville derivative $\riemann t$, which are defined as follows:
\begin{equation*}
	\caputo{t} u (t,x)
	:=\tfrac{1}{\Gamma(1-\alpha)}
	\int_{0}^t
	\dfrac{(\partial_t u)(\tau,x)}{(t-\tau)^{\alpha}}
	\diff \tau
\end{equation*}
and 
\begin{equation*}
	\riemann{t} u (t,x)
 :=  \tfrac{1}{\Gamma(1-\alpha)} \frac{\partial}{\partial t} \left(  \int_0^t \frac{u(\tau,x)}{(t - \tau)^{\alpha}}  d \tau \right).
\end{equation*}
These derivatives can be started at a time $t_0\neq 0$, in which case they are denoted by $\caputo[t_0]{t}$ and $\riemann[t_0]{t}$. Some authors drop $R$ from the Riemann--Liouville derivative, but we will keep it for clarity. 

The initial conditions are a little trickier, as they depend on the type of time derivative. We will explain this below. For the Caputo derivative we have simply
\begin{gather}
	\tag{IC$_C$}
			\label{eq:IC Caputo}
			u(0,x)=u_0(x), \qquad x\in\Omega.
\end{gather}
For the Riemann--Liouville derivative, however, we set the initial condition
\begin{gather}
		\tag{IC$_R$}
		\label{eq:IC Riemann}
		\lim_{h \to 0^+} \riemannOrder{t}{\alpha-1} u (h,x)  =u_0(x), \qquad  x\in\Omega.
\end{gather}
where we define for the range $\alpha - 1 \in (-1,0)$
\begin{equation*} 
	 \riemannOrder{t}{\alpha-1} u (t,x) 
  := 
	 \tfrac{1}{\Gamma (1-\alpha )}   \int_{0}^t (t - \xi )^{-\alpha} u (\xi,x)  d \xi . 
\end{equation*}
This seemingly strange initial condition is motivated and justified by the Laplace transform in \eqref{eq:Riemann Laplace}.

\bigskip 

Using the notation $\bullet \in \{C,R\}$ we will denote the solution above by 
\begin{equation} 
	\label{eq:H definition}
		u(t,x)=\Heat_\bullet[u_0,f,h].
\end{equation}

There have been a significant number of previous works for this family of problems.
The problem with no boundary data  $\mathcal H_C [u_0, f, 0]$ for good data $u_0,f$ has been studied in Gal and Warma in \cite{Gal2020} for general operators $\Ls$. The aim of this paper is to consider jointly the evolution problems with Caputo and Riemann--Liouville time derivative, and exploit the existing duality between them. Moreover, we consider the problems with singular spatial boundary data $h \ne 0$, which is only known for $\alpha = 1$ (see \cite{Chan+GC+Vazquez2020parabolic}).

\bigskip 

Our aim in this paper is to prove existence and uniqueness of suitable solutions of the problem up to finite time. 
Our main results are presented and explained in the next section. In \Cref{thm:h = 0 L2} we prove well-posedness and a representation formula of spectral-type solutions for smooth data $u_0,f$ and $h = 0$. In \Cref{thm:h = 0} we show that this representation is also valid for weighted-integrable data $u_0, f$ and $h = 0$, and provide a weak notion of solution with uniqueness. Lastly, in our main result, \Cref{thm:h ne 0}, we show how the previously introduced functions concentrate towards the boundary to construct solutions of the general case with $h \ne 0$, and give a suitable notion for which they are also unique.

\bigskip 

There has also been progress in other directions. Let us mention that the asymptotic behaviour for $t \to \infty$ in the whole space was considered by \cite{CortazarQuirosWolanski2021,CortazarQuirosWolanski2023Preprint}.

\section{Main results, structure of the paper, and comparison with previous theory} 
\label{sec:main results}

Recalling the theory of elliptic problems, there is a long list of paper dealing with the continuous and bounded solutions of the elliptic problem
\begin{equation}
	\label{eq:elliptic}
	\begin{dcases}
		\Ls U = f, & x \in \Omega, \\ 
		U = 0, & x \in \partial \Omega \, ( \textrm{resp. } \mathbb R^d \setminus \overline \Omega).
	\end{dcases}
\end{equation}
For a general class of integro-differential operators, it is proven \cite{AbatangeloGomezCastroVazquez2019} that there are sequences $f_j $ with support concentrating towards the boundary such that $U_j \to u$, a non-trivial solution to $\Ls u = 0$ in $\Omega$. We pick $u^\star$ a canonical representative of this class, which we will explain below after the statement of \Cref{thm:h = 0}. In the case of the classical Laplacian, one such example is $u^\star = 1$, i.e. one recovers the solution of the non-homogeneous Dirichlet problem.
Letting $\delta(x) = \dist(x, \partial \Omega)$, for the Restricted Fractional Laplacian we recover solutions of the form $u^\star \asymp \delta^{-s}$ whereas for the Spectral Fractional Laplacian $u^\star \asymp \delta^{-2(1-s)}$. 
In \cite{AbatangeloGomezCastroVazquez2019} (see also \cite{abatangelo2015LargeHarmonicFunctions,AbatangeloDupaigne2017}) the authors proved that the additional condition $U/u^* = h$ can be added on the spatial boundary $\partial \Omega$.  
In \cite{Chan+GC+Vazquez2020parabolic} we showed that in this condition can also be added in the parabolic problem
\begin{equation}
	\label{eq:parabolic}
	\begin{dcases}
		\frac{\partial V}{\partial t} + \Ls V = 0, & t> 0 \text{ and }x \in \Omega, \\ 
		V(t,x) = 0, & x \in \partial \Omega \, ( \textrm{resp. } \mathbb R^d \setminus \overline \Omega),\\
		V(0,x) = u_0 (x), & x \in \Omega.
	\end{dcases}
\end{equation}
In this paper we show that non-local-in-time problems also admit the additional singular (or non-singular) boundary condition
\begin{equation}
	\tag{BC}
	\label{eq:singular boundary condition 2}
	\lim_{x \to \zeta} \frac{u(t,x)}{u^\star(x)} = h(\zeta).
\end{equation}

We make the following assumptions on $\Ls$ throughout the paper. We assume that, for every $f \in L^\infty(\Omega)$, \eqref{eq:elliptic} has a unique bounded solution and it is given by integral against the so-called \textit{Green kernel} $\mathbb G$ in the sense that
$$
	U(x) = \int_\Omega \G(x,y) f(y) \diff y.
$$
We denote $\Green[f] = U$.
As in  \cite{Chan+GC+Vazquez2020parabolic}, we will make the following assumptions on $\G$:
\begin{hypothesis}[Fractional structure of the Green function]
~
\begin{itemize}
	\item The Green operator $\Green=\Ls^{-1}$ admits a symmetric kernel $\G(x,y) = \G(y,x)$ with two-sided estimates
	\begin{equation} 
	\tag{G1}
	\label{eq:Green estimates}
	    \G (x,y) \asymp \frac{1}{|x-y|^{d-2s}} \left( 1 \wedge \frac{\delta(x)}{|x-y|} \right)^\gamma  \left( 1 \wedge \frac{\delta(y)}{|x-y|} \right)^\gamma,
	    \qquad x,y\in\Omega,\quad x\neq y,
	\end{equation}
	where $s,\gamma\in(0,1]$.

        \item The Martin kernel, or the $\gamma$-normal derivative of the Green kernel, exists (and therefore enjoys the two-sided estimates):
	\begin{equation}
	\tag{G2}
	\label{eq:Martin estimates}
	    D_\gamma\bG(x,\zeta) 
	    :=
	    \lim_{\Omega\ni y\to\zeta}
	    \frac{\bG(x,y)}{\delta(y)^\gamma}
	    \asymp
	    \frac{\delta(x)^\gamma}{|x-\zeta|^{d-2s+2\gamma}},
	    \qquad x\in\Omega,\quad \zeta\in\partial\Omega,
	\end{equation}
	\item $\Ls$ enjoys the boundary regularity that
	\begin{equation}
	    \tag{G3}
	    \label{eq:G boundary regularity}
	    \Green:\delta^\gamma L^\infty(\Omega) \to \delta^\gamma C(\overline{\Omega}).
	\end{equation}
\end{itemize}
\end{hypothesis}

\begin{remark}[Spectral decomposition]
	\label{rem:spectral}
    Given \eqref{eq:Green estimates}, by the Hille--Yosida theorem $\Ls$ generates a heat semigroup $\Semi(t)$  that solves \eqref{eq:parabolic} (see \cite{BFV2018}). Furthermore, it formally admits an $L^2(\Omega)$ spectral decomposition
	with an orthogonal basis of eigenfunctions $\varphi_j$ with eigenvalues $\lambda_j$
	$$
	\Ls \varphi_j = \lambda_j \varphi_j 
	$$
	The rigorous approach is that $\varphi_j$ form the eigenbasis of $\Green$.
\end{remark}
We make further assumptions on the heat semigroup $\Semi$:
\begin{hypothesis}
	[$\Ls$ generates a submarkovian semigroup $\Semi(t)$]
	\begin{equation}
		\label{eq:submarkovian}
		\tag{S1}
		0 \leq u_0\leq 1
		\implies
		0\leq \Semi(t)[u_0] \leq 1.
	\end{equation}	
\end{hypothesis}
In \cite{Chan+GC+Vazquez2020parabolic}, under these assumptions we have proven that the heat kernel $\bS(t,x,y)$ exists, i.e. for every $u_0 \in L^\infty(\Omega)$ there exists a unique bounded solution of \eqref{eq:parabolic} expressible by
$$
V(t,x) = \int_\Omega \bS(t,x,y) u_0(y) \diff y.
$$
In \cite{Chan+GC+Vazquez2020parabolic} we also proved that $\bS$ admits a $\gamma$-normal derivative, certain estimates near the diagonal of $\bS$ and $D_\gamma \bS$, and a one-sided Weyl-type law for the eigenvalues of $\Ls$. 

\bigskip 

Due to \Cref{rem:spectral}, we can perform the spectral decomposition
\begin{equation}
	\label{eq:spectral decomposition}
	u(t,x) = \sum_{j=1}^\infty u_j(t) \varphi_j(x),
 \qquad
 f(t,x) = \sum_{j=1}^\infty f_j(t) \varphi_j(x),
\end{equation}
for $u_j(t) = \langle u(t,\cdot), \varphi_j \rangle$ and $f_j(t)  = \langle f(t,\cdot) , \varphi_j\rangle$. Thus
\eqref{eq:main} can be rewritten in the eigenbasis as
\begin{equation*}
	\sum_{j = 1}^\infty \Big( \dtfrac t u_j (t) + \lambda_j u_j (t) -  f_j(t) \Big) \varphi_j (x) = 0.
\end{equation*}

We devote \Cref{sec:Caputo and Riemann derivatives} to the study of the ordinary fractional-in-time equations
\begin{equation}
	\label{eq:spectral decomposition ODE}
	\dtfrac t u_j (t) + \lambda_j u_j (t) = f_j (t)
\end{equation}
with the suitable initial conditions.
As in \cite{Gal2020}, this spectral analysis leads to the construction of the kernels
\begin{align}
\label{eq:Salpha from S}
	\bS_\alpha(t,x,y)
	&=\int_0^\infty
	\Phi_\alpha(\tau)
	\bS(\tau t^\alpha,x,y)
	\diff \tau,\\
\label{eq:Palpha from S}
	\bP_\alpha(t,x,y)
	&=\alpha t^{\alpha-1}
	\int_0^\infty
	\tau\Phi_\alpha(\tau)
	\bS(\tau t^\alpha,x,y)
	\diff \tau,
\end{align}
where $\Phi_\alpha$ is the well-known Mainardi function given in \eqref{eq:Mainardi}. 
The associated integral operators are:
\begin{align*}
	\cS_\alpha(t)[u_0](x)
	&=\int_{\Omega}
	\bS_\alpha(t,x,y)
	u_0(y)
	\diff y, 
	\qquad 
	\cP_\alpha(t)[f](x)
	=\int_{\Omega}
	\bP_\alpha(t,x,y)
	f(y)
	\diff y.
\end{align*}
This analysis works for $h = 0$. To deal with the case of non-trivial singular boundary data $h \ne 0$ we need to introduce the notation for the $\gamma$-normal derivatives of $\cP_\alpha$ and $\bP_\alpha$:
\begin{align*}
	D_\gamma\cP_\alpha[h](x)
	&=\int_{\partial\Omega}
	D_\gamma\bP_\alpha(t,x,\zeta)
	h(\zeta)
	\diff \zeta,\\
	D_\gamma\bP_\alpha(t,x,\zeta)
	&=\alpha t^{\alpha-1}
	D_\gamma
    \bigg[\int_0^\infty
	\tau\Phi_\alpha(\tau)
	\bS(\tau t^\alpha,x,\cdot)
	\diff \tau
    \bigg](\zeta).
\end{align*}
Finally, we propose as solution to  \eqref{eq:main} the function $u = \Heat_{\bullet}[u_0,f,h]$ given by
\begin{equation} 
	\label{eq:H representation}
\begin{aligned}
	\Heat_C[u_0,f,h](t)
	&=\Semi_\alpha(t)[u_0]
	+\int_0^t\cP_\alpha(t-\tau)[f(\tau)]\diff\tau
	+\int_0^t D_\gamma\cP_\alpha(t-\tau)[h(\tau)]\diff\tau,\\
	\Heat_R[u_0,f,h](t)
	&=\cP_\alpha(t)[u_0]
	+\int_0^t\cP_\alpha(t-\tau)[f(\tau)]\diff\tau
	+\int_0^t D_\gamma\cP_\alpha(t-\tau)[h(\tau)]\diff\tau.
\end{aligned}
\end{equation} 
Notice that the choice of Caputo or Riemann--Liouville derivative only affects the initial condition, in the sense that
$$\Heat_C[0,f,h] = \Heat_R[0,f,h].$$
Hence, when $u_0  = 0$ we drop the sub-index $C$ or $R$ and denote simply $u = \Heat[0,f,h]$. 
We point out that the super-position principle (i.e. linearity) means that
$$
\Heat_\bullet [u_0,f,h] = 	\Heat_\bullet [u_0,0,0] + 	\Heat [0,f,0] + 	\Heat [0,0,h].
$$

We make some further technical assumptions which are needed below in this paper:
\begin{hypothesis}[Off-diagonal bound on the heat kernel $\bS$]
        \begin{equation}
        \tag{S2}
        \label{eq:heat-off-diag}
        \frac{\bS(t,x,y)}{\delta(x)^\gamma}
        \leq C(\eps)
         \quad \text{ for } \quad
        |x-y| \geq \eps,   t \ge 0.
        \end{equation}
\end{hypothesis} 
\begin{hypothesis}[Uniform exchange of limits between integral and $D_\gamma$] We assume that $\bP_{\alpha}$ has the following properties:
    \begin{equation}
        \tag{E}
        \label{eq:exchange}
        \begin{aligned} 
        D_\gamma \mathbb P_\alpha (t,x, \zeta) &= \alpha t^{\alpha-1} \int_{0}^{\infty} \tau \Phi_\alpha(\tau)  D_\gamma  \p  (\tau t^{\alpha},x,\zeta)\diff \tau
        \\ 
        D_\gamma \mathbb G(x,\zeta) &= \int_0^\infty D_\gamma \p(t,x,\zeta) \diff t.
        \end{aligned} 
    \end{equation}
    where $\Phi_\alpha$ is the Mainardi function given in \eqref{eq:Mainardi}. 
\end{hypothesis}

\begin{remark} 
We remark that \eqref{eq:Green estimates} implies 
(see \cite{AbatangeloGomezCastroVazquez2019}) that
\begin{equation}
    \label{eq:Green AGV}
    \sup_{x\in\Omega}
    \int_\Omega
    \bigg(
        \frac{\bG(x,y)}{\delta(x)^\gamma}
        \delta(y)^\gamma
    \bigg)^p
    \diff y
\leq C,
\end{equation}
for some $p>1$. Moreover, under \eqref{eq:Martin estimates} and \eqref{eq:exchange},
\[
u^{\star}(x)
=\int_{\partial \Omega}  D_\gamma \mathbb G(x, \zeta) \diff \zeta
=\int_0^\infty \int_{\partial \Omega}  D_\gamma \p (\tau,x, \zeta) \diff \zeta \diff \tau.
\]
\end{remark}

In order to develop a theory of \emph{classical} boundary singular solutions for time-fractional equations, we impose the following extra hypothesis:
\begin{hypothesis} [Uniform control of the time tail of $D_\gamma \bS$ near $\partial \Omega$]
\begin{equation}
    \tag{S3}
    \label{eq:H**}
    \text{For any $\delta > 0$ and $\zeta_0 \in \partial \Omega$, } \lim_{x \to \zeta_0} \int_0^\infty
	\Phi_\alpha(\tau) 
	\dfrac
	{  \int_{\tau \delta^\alpha}^\infty  \int_{\partial \Omega}   D_\gamma  \p  (\sigma,x,\zeta) \diff \zeta \diff \sigma }
	{  \int_0^\infty \int_{\partial \Omega}  D_\gamma \p (\tilde\tau,x, \tilde \zeta) \diff \tilde \zeta \diff \tilde\tau  } 
	\diff \tau = 0 .
\end{equation}
\end{hypothesis}

\bigskip 

In \Cref{sec:h=0 L2} we develop the $L^2$ theory using spectral analysis. For this we define the natural energy space
$$
	\mathrm H (\Omega) = \left \{   u \in L^2 (\Omega) : \| \Ls u \|_{L^2(\Omega)}^2 = \sum_{j=1}^{\infty} \lambda_j^2 \langle u, \varphi_j \rangle^2 < \infty   \right \}.
$$
The well-posedness result is the following.
\begin{theorem}
	\label{thm:h = 0 L2}
	Assume \eqref{eq:Green estimates}, \eqref{eq:Martin estimates}, \eqref{eq:G boundary regularity}, and \eqref{eq:submarkovian}, $u_0 \in \mathrm H (\Omega)$, $f \in C^1([0,T]; L^2 (\Omega))$, and $h = 0$. Then, 
	\begin{itemize}
		\item Caputo derivative case: There is a unique function 
		$$
			u \in C([0,T]; L^2(\Omega)) \cap C((0,T]; \mathrm H (\Omega)) \cap C^1((0,T];L^2(\Omega))
		$$ 
		that satisfies of (P$_C$) in the spectral sense \eqref{eq:spectral decomposition ODE} together with the initial condition \eqref{eq:IC Caputo}, i.e. $u(0,\cdot) = u_0$.

		\item Riemann--Liouville derivative case: There is a unique function $u$ such that 
		$$
			t^{1-\alpha} u \in C([0,T]; L^2(\Omega)) \cap C((0,T]; \mathrm H (\Omega)) \cap C^1((0,T];L^2(\Omega)),
		$$
		and $u$ satisfies of (P$_R$) in the spectral sense \eqref{eq:spectral decomposition ODE} together with the initial condition \eqref{eq:IC Riemann}.
	\end{itemize}
	In each case, this solution is given by $u=\Heat_{\bullet}[u_0,f,0]$ as in \eqref{eq:H representation}. 
\end{theorem}
The case of Caputo derivative is already covered in \cite[Theorem 2.1.7]{Gal2020} in a slightly different functional setting.

\bigskip 

In \Cref{sec:h=0 L1 and inf} we extend this theory for $h = 0$ outside the $L^2$ framework. To this end, we define a generalised notion of solution in the ``optimal'' domain of the Green kernel (the weighted space $L^1(\Omega, \delta^\gamma)$) and we prove uniform integrability estimates. 
We provide the following definition of weak-dual solution, which we justify by the ``duality'' between the Caputo and Riemann-Liouville derivatives presented in \Cref{sec:Caputo-Riemann duality}.
\begin{definition}
	\label{def:weak-dual}
	\begin{itemize} 
	\item We say that $u$ is weak-dual solution of (P$_C$) if $u \in L^1(0,T;L^1(\Omega,\delta^\gamma))$ and
    \begin{equation}
    \label{eq:very weak h = 0}
	\begin{aligned}
		\int_{0}^{T}\int_{\Omega}
		u(t,x)\phi(T-t,x)
		\diff x\diff t
		&=\int_\Omega u_0 (x) \Big(\riemannOrder{t}{\alpha-1} \mathcal H [0,\phi,0]\Big)(T,x) \diff x\\
		&\qquad
		+\int_0^T \int_\Omega f(t,x) \Heat [0,\phi,0](T-t,x) \, \diff x \diff t\\
		&\qquad
		+\int_0^T \int_{\partial \Omega} h(t, \zeta) D_\gamma \mathcal H[0,\phi,0] (T-t, \zeta) \diff \zeta \diff t,
	\end{aligned}
        \end{equation}
	for all $\phi\in \delta^\gamma L^\infty((0,T)\times\Omega)$. 
	
	\item We say that $u$ is weak-dual solution of (P$_R$) if $u \in L^1(0,T;L^1(\Omega,\delta^\gamma))$ and
	\begin{align*}
		\int_{0}^{T}\int_{\Omega}
		u(t,x)\phi(T-t,x)
		\diff x\diff t
		&=\int_\Omega u_0 (x) \cH[0,\phi,0](T,x) \diff x\\
		&\qquad
		+\int_0^T \int_\Omega f(t,x) \Heat [0,\phi,0](T-t,x) \, \diff x \diff t\\
		&\qquad
		+\int_0^T \int_{\partial \Omega} h(t, \zeta) D_\gamma \mathcal H[0,\phi,0] (T-t, \zeta) \diff \zeta \diff t.
	\end{align*}
 	for all $\phi\in \delta^\gamma L^\infty((0,T)\times\Omega)$. 

	\end{itemize} 
\end{definition}
For this notion of solution, we provide a well-posedness result:
\begin{theorem}
	\label{thm:h = 0}
Assume \eqref{eq:Green estimates}, \eqref{eq:Martin estimates}, \eqref{eq:G boundary regularity}, and \eqref{eq:submarkovian}, $u_0\in L^1(\Omega,\delta^\gamma)$, $f\in L^1(0,T;L^1(\Omega,\delta^\gamma))$, and $h= 0$. Then, \eqref{eq:main} admits a unique weak-dual solution and it is given by $u=\Heat_{\bullet}[u_0,f,0]$ as in \eqref{eq:H representation}. 
\end{theorem} 

Finally, in \Cref{sec:h ne 0} we ``concentrate'' $f$ towards $\partial \Omega$ to construct solutions with $h \ne 0$. 
In \cite{AbatangeloGomezCastroVazquez2019} the authors construct a singular solution as follows. As the authors did in \cite{AbatangeloGomezCastroVazquez2019,Chan+GC+Vazquez2020parabolic} we define $A_j \defeq \{ 1/j < \delta(x) < 2/j\}$ and
\begin{equation*}
	f_j (x) \defeq \frac{|\partial \Omega|}{|A_j|} \frac{\chi_{A_j}}{\delta(x)^\gamma} .
\end{equation*}
Under our assumptions, it then allows to show that
\begin{equation*}
	\Ls^{-1} [f_j ] \to u^\star \qquad \text{ in } L^1_{\rm loc} (\Omega).
\end{equation*}
These canonical solutions have, in some sense, ``uniform'' boundary conditions. Under the assumption \eqref{eq:Green estimates}, the boundary blow-up rate is given by \cite[equation (4.2)]{AbatangeloGomezCastroVazquez2019}, namely
\begin{equation*}
	u^\star (x) \asymp
	\begin{dcases}
	\delta(x)^{2s-\gamma -1 },
	    & \gamma>s-\frac12,\\
	\delta(x)^\gamma(1+|\log\delta(x)|),
	    & \gamma=s-\frac12,\\
	\delta(x)^\gamma,
	    & \gamma<s-\frac12.
	\end{dcases}
\end{equation*}
Notice that in the classical case $\gamma = s = 1$ so we recover $\delta^0$. 
In particular, when $\Ls = -\Delta$, this yields that $u^\star \equiv 1$, the only solution of $-\Delta U = 0$ in $\Omega$ such that $u = 1$ on $\partial \Omega$. 
Passing to the limit in the weak formulation
\begin{equation*}
	\int_\Omega u_j \psi = \int_\Omega f_j \Ls^{-1}[\psi], \qquad \forall \psi \in L^\infty_{\rm c} (\Omega),
\end{equation*}
we recover that $u^\star$ satisfies the very weak formulation 
\begin{equation*}
	\int_\Omega u^\star \psi = \int_{\partial \Omega} \lim_{x \to \zeta} \frac {\Ls^{-1}[\psi](x)}{\delta(x)^\gamma}  \qquad \forall \psi \in L^\infty_{\rm c} (\Omega).
\end{equation*}
More general solutions are constructed by letting
\begin{equation*}
	f_j (x) \defeq \frac{|\partial \Omega|}{|A_j|} \frac{\chi_{A_j}}{\delta(x)^\gamma} h(P_{\partial \Omega} (x)),
\end{equation*}
where $P_{\partial \Omega}$ is the orthogonal projection onto the boundary, which is well-defined in $A_j$ for $j$ large enough and the very weak formulation is 
\begin{equation*}
	\int_\Omega u^\star \psi = \int_{\partial \Omega} \lim_{x \to \zeta} h (\zeta) \frac {\Ls^{-1}[\psi](x)}{\delta(x)^\gamma}  \qquad \forall \psi \in L^\infty_{\rm c} (\Omega).
\end{equation*}

For the parabolic case when $\alpha = 1$, the same idea of picking
\begin{equation}
    \label{eq:fj to h}
	f_j (t,x) \defeq \frac{|\partial \Omega|}{|A_j|} \frac{\chi_{A_j}}{\delta(x)^\gamma} h(t,P_{\partial \Omega} (x)).
\end{equation}
was shown to work in \cite{Chan+GC+Vazquez2020parabolic}. Now try to extend to $\alpha \in (0,1)$ to construct $\mathcal H_\bullet [0,0,h]$.
\color{black}
Due to \Cref{rem:equality u_0 = 0}, it is clear that $\mathcal H[0,0,h]$ will be the same for both fractional time derivatives.

\begin{theorem}
		\label{thm:h ne 0}
	Assume \eqref{eq:Green estimates}, \eqref{eq:Martin estimates}, \eqref{eq:G boundary regularity}, \eqref{eq:submarkovian}, \eqref{eq:heat-off-diag} and \eqref{eq:exchange}. 
	For any $u_0, f = 0$, and $h\in L^1((0,T)\times \partial\Omega)$. Then, 
	\begin{enumerate}[\rm i)]

    \item Let $f_j$ be given by \eqref{eq:fj to h}, then $\mathcal H[0,f_j,0] \to \mathcal H[0,0,h]$.
 
	\item 
	\eqref{eq:main} admits a unique weak-dual solution and it is given by $u=\Heat[0,0,h]$ as in \eqref{eq:H representation}.

	\item Assume, in addition, \eqref{eq:H**} and $h\in C(\partial\Omega)$. Then \eqref{eq:singular boundary condition 2} holds in the sense that, for each $\zeta \in \partial \Omega$,
	\begin{equation}
		\label{eq:singular boundary condition}
		\lim_{x \to \zeta} \frac{\mathcal H[0,0,h] (t,x) }{ u^\star (x) } = h(t,\zeta).
	\end{equation}
	\end{enumerate} 
\end{theorem}

\paragraph{A comment on the hypothesis}
In the previous works, we have checked the hypotheses \eqref{eq:Green estimates},
\eqref{eq:Martin estimates},
\eqref{eq:G boundary regularity},
\eqref{eq:submarkovian} in the examples given in \Cref{sec:heat kernel examples}. It is not difficult to check that the new hypotheses \eqref{eq:heat-off-diag},
\eqref{eq:exchange} also hold in those cases.

\section{The Caputo and Riemann--Liouville time derivatives}
\label{sec:Caputo and Riemann derivatives}

	\subsection{The Riemann--Liouville integral}

The Riemann--Liouville integral is defined for $\alpha > 0$ by
\begin{equation*}
	\RLInt^\alpha w(t) 
 := \frac{1}{\Gamma(\alpha)} \int_0^t w(\xi) (t - \xi)^{\alpha - 1} \diff \xi .
\end{equation*}
This operator is continuous from $L^1(0,T) \to L^1(0,T)$. It has the derivative-like properties 
\begin{equation*}
	\frac{\diff}{\diff t} \RLInt^{\alpha + 1} w(x) = \RLInt^\alpha w(x), \qquad \RLInt^\alpha \RLInt^\beta w = \RLInt^{\alpha + \beta} w
\end{equation*}
and its Laplace transform, which we denote here by $\Laplace$, is given by
\begin{equation*}
	\Laplace[\RLInt^\alpha w](s) = s^{-\alpha} \Laplace[w](s)
\end{equation*}
whenever $\Re(s) > \sigma$ and $w(t) e^{-\sigma t} \in L^1(0,\infty)$.

\bigskip 

Given these definitions, we point out that the Caputo derivative can be equivalently defined for $\alpha \in (0,1)$ by
\begin{equation*}
	\caputo t w (t) := \RLInt^{1-\alpha} \frac{\diff}{\diff t} w.
\end{equation*}
This formula can be extended to $\alpha > 1$. On the other hand, the Riemann--Liouville derivative is equivalently defined by
\begin{equation*}
	\riemann t w
	=\begin{dcases}
		\frac{\diff^{\ceil \alpha}}{\diff t^{\ceil \alpha}} \RLInt^{\ceil \alpha - \alpha }w, & \text{if }\alpha > 0 , \\
		w, & \text{if } \alpha = 0, \\
		\RLInt^{-\alpha} w, & \text{if } \alpha <0,
	\end{dcases}
\end{equation*}
where $\lceil \alpha \rceil$ is the ceiling function of $\alpha$.

\bigskip 

Due to this representation, the Laplace transform of these differentiation operators can be easily computed. Indeed, for the Caputo derivative we have that
\begin{equation}
	\label{eq:Caputo Laplace}
	\Laplace \left[ \caputo t w \right] (s) = s^\alpha \Laplace [w] - s^{\alpha - 1} w(0), \qquad \alpha \in (0,1).
\end{equation}
Similarly, the Laplace transform of the Riemann--Liouville derivative is given by
\begin{equation}
	\label{eq:Riemann Laplace}
	\Laplace\left[ \riemann t w \right] (s) = s^\alpha \Laplace[w] (s) - \lim_{h\to 0^+} [ \riemannOrder t {\alpha-1} w(h) ], \qquad \alpha \in (0,1).
\end{equation}
	
	\subsection{The Mittag-Leffler and Mainardi functions} 
For the solution of these equations we will use the Mittag-Leffler functions defined by
\begin{equation}\label{eq:ML-series}
	E_{\alpha,\beta} (z)
	=\sum_{k=0}^{\infty}
	\dfrac{z^k}{\Gamma(\alpha k+\beta)} .
\end{equation}
It is also common to denote $E_\alpha = E_{\alpha,1}$. 
These are entire functions.
Notice, furthermore, that $E_{1,1} (z) = e^z$. The advantage of Mittag-Leffler function in the study of \eqref{eq:spectral decomposition ODE} can be seen in the following computation concerning the Laplace transform of its suitable moment:
\begin{equation}
	\label{eq:Mittag-Leffler Laplace}
	\Laplace[t^{\beta-1}E_{\alpha,\beta}(-\lambda t^\alpha)](s)
	=
	\int_0^\infty 
	t^{\beta - 1} E_{\alpha, \beta} (-\lambda t^\alpha) e^{-st} \diff t = \frac{s^{\alpha-\beta}}{s^\alpha +\lambda},
\end{equation}
so long as $\Re(s), \Re(\alpha), \Re(\beta),\lambda  > 0$. 
The cases $\beta=1$ and $\beta=\alpha$ will be particularly useful.

There are many known properties of $E_\alpha(-t^\alpha)$ (see, e.g. \cite{Mainardi2014} and the references therein). In particular
\begin{align*}
	E_\alpha(-t^\alpha) 
    &\sim f_\alpha(t)
	\coloneqq \frac{1}{
        1 + \Gamma(1+\alpha)^{-1} \, t^\alpha
    },  \qquad \text{as } t \to 0, \\
	E_\alpha(-t^\alpha) &\sim g_\alpha(t) 
	\coloneqq  \frac{1}{1+\Gamma(1-\alpha)\,t^\alpha},
	\qquad \text{as } t \to \infty. 
\end{align*}
It is left as a conjecture that $f_\alpha(t) \le E_\alpha(-t^\alpha) \le g_\alpha(t)$. 
Also, it is known that $E_\alpha(-t^\alpha)$ is completely monotone in $t$, i.e. %
$(-1)^n \frac{\diff^n }{dt^n} E_\alpha(-t^\alpha) \ge 0$. From \eqref{eq:ML-series}, it is easily verified that $E_{\alpha,\alpha}(-t^\alpha)=t^{1-\alpha}\,(-1)\frac{d}{dt}E_{\alpha}(-t^{\alpha})$ is also non-negative and non-increasing. By manipulating the series, it is easy to see the recurrence property
\[
zE_{\alpha,\beta}(z)=E_{\alpha,\beta-\alpha}(z)-\frac{1}{\beta-\alpha}.
\]
This implies, in particular, the global bounds
\begin{equation}\label{eq:ML-bound}
E_{\alpha}(-\lambda t^\alpha)\leq \frac{C}{1+\Gamma(1-\alpha)\lambda t^\alpha},
    \quad
E_{\alpha,\alpha}(-\lambda t^\alpha)\leq \frac{C}{1+\Gamma(-\alpha)(\lambda t^\alpha)^2},
    \quad
\text{for } t\geq 0.
\end{equation}

\bigskip 

We also recall the definition of the Mainardi function (or Wright-type function)
\begin{equation}
\label{eq:Mainardi}
\Phi_\alpha(t)
\coloneqq\sum_{k=0}^{\infty}
\dfrac{
	(-t)^k
}{
	k!\Gamma(1-\alpha(k+1))
},
\qquad \text{for }t>0.
\end{equation}
It is known that $\Phi_\alpha \ge 0$ (see \cite[Section 4]{mainardi2010MWrightFunctionTimeFractional}).
It is in fact an entire function for complex arguments $t\in \mathbb{C}$, and has explicit moments
\begin{equation}\label{eq:Wright-moment}
	\int_0^\infty
	t^p
	\Phi_\alpha(t)
	\,dt
	=\dfrac{
		\Gamma(1+p)
	}{
		\Gamma(1+\alpha p)
	},
	\qquad
	p>-1.
\end{equation}
In particular, $\Phi_\alpha$ is a probability density function. We observe that $\Phi_\alpha$ arises as the inverse Laplace transform of the Mittag-Leffler function. In fact, we have the following relations which are easily verified via a series expansion:
\begin{align} 
	\label{eq:Wright and Mittag-Leffler 1}
	\int_{0}^{\infty}
	\Phi_\alpha(t)
	e^{-t z}
	\diff t
	&= E_\alpha(-z), \\
	\label{eq:Wright and Mittag-Leffler 2}
	\int_{0}^{\infty}
	t \Phi_\alpha(t)
	e^{-t z} \diff t
	&= - \frac{\diff E_\alpha}{\diff z}(-z)
	=\frac{1}{\alpha}E_{\alpha,\alpha}(-z).
\end{align}

\subsection{Ordinary integro-differential equations with Caputo derivative}

We focus our attention on 
\begin{gather}
	\tag{ODE$_C$}
	\label{eq:ODE Caputo}
	\begin{dcases} 
		\caputo t u (t) + \lambda u (t) = f (t), & t > 0, \\
		u (0) = u_{0}.
	\end{dcases} 
\end{gather}
Applying \eqref{eq:Caputo Laplace}, we can find the solution of \eqref{eq:ODE Caputo} in the Laplace variable:
\begin{align*}
	\Laplace [u] (s) 
	&= \frac{s^{\alpha-1}}{s^\alpha + \lambda} u_0 + \frac{1}{s^\alpha + \lambda} \Laplace[f](s)\\
	&=\Laplace[u_0 E_\alpha(-\lambda t^\alpha)](s)
	+\Laplace[t^{\alpha-1}E_{\alpha,\alpha}(-\lambda t^\alpha)](s)
	\Laplace[f](s),
\end{align*}
where in the last equality we have \eqref{eq:Mittag-Leffler Laplace} with $\beta=1$ and $\beta=\alpha$. Taking inverse Laplace transform we easily obtain the general solution
\begin{equation}
	\label{eq:ODE Caputo solution} 
	u (t) = u_0 E_\alpha (-\lambda t^\alpha) + \int_0^t P_{\alpha}(t-\tau;\lambda) f(\tau) \diff \tau.
\end{equation}
where
\begin{equation}
	\label{eq:P alpha}
	P_{\alpha}(t;\lambda) = t^{\alpha-1} E_{\alpha,\alpha}(-\lambda t^\alpha)
	=-\frac{1}{\lambda}\frac{d}{dt} E_\alpha(-\lambda t^\alpha)
        = \frac{\alpha t^{\alpha-1}}{\lambda} \int_{0}^{\infty}
	\sigma \Phi_\alpha(\sigma)
	e^{-\sigma \lambda t^\alpha} \diff \sigma.
\end{equation}
This has been discussed in \cite{Kilbas2006}.
As the product of non-negative and non-increasing functions, $P_\alpha (\cdot, \lambda)$ is non-negative and non-increasing for all $\lambda \ge 0$.

\subsection[Ordinary integro-differential equations with Riemann--Liouville derivative]{Ordinary integro-differential equations with \\ Rie\-mann--Liou\-ville derivative}

Consider now
\begin{gather} 
	\begin{dcases}
		\tag{ODE$_R$}
		\label{eq:ODE Riemann}
		\riemann t v(t)  + \lambda v(t)  = g(t) , & t > 0 , \\		
		\lim_{h \to 0^+} \riemannOrder t {\alpha-1} v (h) = v_0.
	\end{dcases}
\end{gather} 
Applying \eqref{eq:Riemann Laplace}, we can solve the ODE in the Laplace space as
\begin{equation*}
	\Laplace [v] (s) = \frac{1}{ s^\alpha + \lambda}v_0 + \frac{\Laplace[g]}{s^\alpha + \lambda}.
\end{equation*}
Therefore, the general solution for \eqref{eq:ODE Riemann} is given by
\begin{equation}
	\label{eq:ODE Riemann solution}
	v(t) =   v_0 P_{\alpha}(t;\lambda)+ \int_0^t P_{\alpha} (t-\tau;\lambda) g(\tau) \diff \tau .
\end{equation} 
where $P_{\alpha}(\cdot;\lambda)$, given by \eqref{eq:P alpha}, is the same as in the solution for \eqref{eq:ODE Caputo}. 

\begin{remark} 
	Notice that if $u$ is the solution of \eqref{eq:ODE Caputo} and $v$ is the solution of \eqref{eq:ODE Riemann} with $u_0 = v_0 = 0$ and $f = g$, then $u \equiv v$.
\end{remark} 

\begin{remark}
	\label{rem:Palpha singular at 0}
	Notice that providing $P_\alpha(t;\lambda) \sim \frac{t^{\alpha-1}}{\Gamma(\alpha)}$ as $t \to 0$. Therefore, the initial condition $v(0^+)$ must be understood in a singular way.
\end{remark}

\subsection{Integration by parts with the Caputo derivative}

To compute the adjoint of the Caputo derivative in $L^2(0,T)$ we have
$$
\begin{aligned} 
	\int_0^T  \varphi(t)  \left(\caputo t [u] \right) (t)   \diff t  
	&=	 \frac{1}{\Gamma(1-\alpha)}
	\int_0^T \varphi (t) \int_0^t (t - \sigma)^{-\alpha} u' (\sigma) \diff \sigma \diff t \\
	&=
	\frac{1}{\Gamma(1-\alpha)}  \int_0^T u' (\sigma)  \int_\sigma^T (t - \sigma)^{-\alpha} \varphi(t) \diff t \diff \sigma \\
	&=
	\int_0^T u (\sigma) { \frac{-1}{\Gamma(1-\alpha)}  \frac{\partial}{\partial \sigma}  \left(  \int_\sigma^T (t - \sigma)^{-\alpha} \varphi(t)  \diff  t \right) } \diff  \sigma \\
	&\quad 	+\frac{1}{\Gamma(1-\alpha)} \left( u(T) \lim_{\sigma \to T^-} \int_{\sigma}^T (t-\sigma)^{-\alpha} \varphi(t) \diff t   - u(0)    \int_0^T t^{-\alpha} \varphi(t)  \diff  t     \right).  
\end{aligned}
$$
Since the equation with Caputo derivative involves $\left(\caputo t [u] \right) (t)$ and $u(0)$, we are therefore interested in the adjoint problem
\begin{equation*}
	\begin{dcases}
		\left(\caputo t \right)^* [\varphi]  (\sigma) \defeq \frac{-1}{\Gamma(1-\alpha)}  \frac{\partial}{\partial \sigma}  \left(  \int_\sigma^T (t - \sigma)^{-\alpha} \varphi(t)  \diff  t \right) \\ 
		\varphi_T \defeq \tfrac{1}{\Gamma(1-\alpha)} \lim_{\sigma \to T^-} \int_{\sigma}^T (t-\sigma)^{-\alpha} \varphi(t)  \diff  t.
	\end{dcases}
\end{equation*}
This is an integro-differential equation involving the right Riemann--Liouville derivative (with final condition given by a fractional Riemann--Liouville integral of order $1-\alpha$).  %
Nevertheless, as for the case $\alpha =1$, we do not expect $\left(\caputo t \right)^* [\varphi] + \Ls \varphi = 0$ to have a solution, so we ``reverse time''.

\subsection{Caputo and Riemann--Liouville derivatives are adjoint up to reversing time}
\label{sec:Caputo-Riemann duality}

As usual, we want to reverse time $\varphi (t) = \phi (T-t)$ so that, taking $\tau = T - \sigma$, we have
\begin{align*}
	(\caputo{t} )^* [ \varphi ] (\sigma) 
	&= \frac{-1}{\Gamma(1-\alpha)}  \frac{\partial}{\partial \sigma}  \left(  \int_\sigma^T (t - \sigma)^{-\alpha} \varphi(t)  \diff t \right) 
	\\
	&
	= \frac{-1}{\Gamma(1-\alpha)}  \frac{\partial}{\partial \sigma}  \left(  \int_\sigma^T (t - \sigma)^{-\alpha} \phi(T-t) \diff t \right) \\
	&= \frac{-1}{\Gamma(1-\alpha)}  \frac{\partial}{\partial \sigma}  \left(  \int_0^{T-\sigma} (T-\xi  - \sigma)^{-\alpha} \phi(\xi )  \diff \xi \right) 
	\\
	&
	= \frac{1}{\Gamma(1-\alpha)}  \frac{\partial}{\partial \tau}  \left(  \int_0^{\tau} (\tau -\xi  )^{-\alpha} \phi(\xi )  \diff \xi \right) \\
	&= \left(\riemann \tau\right) [\phi] (\tau)
	= \left(\riemann \tau\right) [\phi] (T - \sigma).
\end{align*}
This is precisely the (left) Riemann--Liouville fractional derivative. 
Notice the ``initial'' conditions
\begin{align*}
	\frac{1}{\Gamma (1-\alpha )} \lim_{\sigma \to T^-} \int_{\sigma}^T (t-\sigma)^{-\alpha} \varphi (t)  \diff t 
	&=\frac{1}{\Gamma (1-\alpha )}  \lim_{h \to 0^+} \int_{0}^h (h - \xi )^{-\alpha} \phi (\xi)  \diff \xi \\
	&=\lim_{h \to 0^+} \riemannOrder{t}{\alpha-1} [\phi] (h).
\end{align*}

But then we can rewrite the integration by parts formula as
\begin{equation} 
	\begin{aligned} 
		\int_0^T  &\phi(T-t)  \left(\caputo t [u] \right) (t)   \diff t  
		+ u(0) \int_0^T \frac{ t^{-\alpha} \phi(T-t) }{\Gamma(1-\alpha)}  \diff  t \\   
		&=  \int_0^T u (t) 
		\left(\riemann \tau[\phi]\right)  (T - t) %
		\diff t 
		+ u(T) \lim_{h \to 0^+} \riemannOrder{t}{\alpha-1} [\phi] (h). 
	\end{aligned}
\end{equation} 

Thus, if $u$ solves \eqref{eq:ODE Caputo} and $v$ solves \eqref{eq:ODE Riemann}, then we have that
\begin{equation}
	\label{eq:ODE adjoint}
	\begin{aligned} 
		\int_0^T v(T-t) &\Big( - \lambda u(t) + f(t) \Big) \diff t + u_0 \int_0^T \frac{ t^{-\alpha} v(T-t)}{\Gamma(1-\alpha)} \diff t \\
		&= \int_0^T u(t) \Big( - \lambda v (T-t) + g(T-t) \Big) \diff t + u(T) v_0.
	\end{aligned}
\end{equation}
Notice that the only remainder that we have due to ``non-locality'' is the second term on the left-hand side. As $\alpha \to 1$, we recover %
the classical integration by parts.

\section{Time-fractional problem when \texorpdfstring{$h=0$}{h=0}. An \texorpdfstring{$L^2$}{L2} theory}
\label{sec:h=0 L2}

We now go back to the spectral decomposition \eqref{eq:spectral decomposition} and take advantage of the explicit solutions of \eqref{eq:spectral decomposition ODE} (as \eqref{eq:ODE Caputo solution} and \eqref{eq:ODE Riemann solution}).

Given suitable functions $F : \mathbb R \to \mathbb R$ and a spectral decomposition of $\Ls$, it is natural to define  $F(\Ls)$ by the linear operator such that $F(\Ls) [\varphi_j] = F(\lambda_j) \varphi_j$. Therefore
\begin{equation}
	\label{eq:FL}
	F(\Ls) [u] (x) = \sum_{j=1}^\infty F(\lambda_j) \langle u , \varphi_j \rangle \varphi_j (x). 
\end{equation}
Recall that the solution of the local-in-time heat equation $\partial_t u = - \Ls u_0$ is given by 
\[
\Semi(t)v
=\sum_{m=1}^{\infty}
e^{-\lambda_m t}
\angles{v,\varphi_m}
\varphi_m = e^{-t \Ls} v.
\]
Hence, we have the kernel representation 
\begin{equation*}
	\Semi(t)v (x) = \int_\Omega \p (t,x,y) u_0 (y) \diff y, \qquad \p(t,x,y) = \sum_{j=0}^\infty e^{-\lambda_j t} \varphi_j (x) \varphi_j (y).
\end{equation*}
In various examples we know upper and lower bounds for $\p$.

\normalcolor

Recall that, in this notation, the solution of the elliptic problem $\Ls^{-1} [f]$ is given precisely by 
\begin{align*}
	\Ls^{-1} [f]  (x)
	&= 
	\sum_{j=1}^\infty 
		\frac { 1 } {\lambda_j }
		\angles{f, \varphi_j}
		\varphi_j   (x)
	= 
	\int_0^\infty \sum_{j=1} ^\infty e^{-\lambda_j t} 	\angles{f, \varphi_j} 	\varphi_j  (x)  \diff t \\
	&=
	\int_0^\infty e^{-t \Ls} [f] (x) \diff t 
\end{align*}
Or, equivalently written as in terms of the kernels
\begin{equation*}
	\G(x,y) = \int_0^\infty \p(t,x,y) \diff t.
\end{equation*}

\begin{remark}
	Notice that the solution of $\Ls u + \mu u = f$ can be obtained similarly by
	\begin{equation*}
		(\Ls + \mu )^{-1} = \int_0^\infty e^{-\mu t} e^{-t \Ls} \diff t.
	\end{equation*}
	This is well defined for $\mu > -\lambda_1$. Since the heat semigroup is non-negative, for $\mu \in (-\lambda_1, \infty)$, for $0 \le f \in L^2 (\Omega)$ we construct exactly one non-negative solutions in $L^2 (\Omega)$. This is the ethos behind \cite{Chan+GC+Vazquez2020parabolic}.
\end{remark}

\color{black}

	\subsection{Caputo}

We would like to obtain a representation formula for $\Heat_{C}[u_0,f,0]$. We give two equivalent expressions: one in terms of Mittag-Leffler function and through \eqref{eq:FL}, and the other in terms of Mainardi function and the heat kernel $\p$.

Due to \eqref{eq:spectral decomposition ODE} and \eqref{eq:ODE Caputo solution}, it is clear, from the solution of the coefficients of the spectral decomposition, that we can write
\begin{align*}
	\mathcal H_C[u_0,0,0] (t,x) &= \sum_{j=1}^\infty u_j(t) \varphi_j (x) = \sum_{j=1}^\infty E_{\alpha}(-\lambda_j t^\alpha ) \langle u_0, \varphi_j \rangle \, \varphi_j (x) \\
	&= E_\alpha (-t^\alpha \Ls) u_0.
\end{align*}
We denote this operator by $\Semi_\alpha (t) \defeq E_\alpha (-t^\alpha \Ls)$. 
Using a similar argument, it can be deduced that the solution of ($P_C$) is given by
\begin{align}
\label{eq:HC}
	\mathcal H_C [u_0,f,0] (t)
	&=\Semi_\alpha(t)u_0
	+\int_0^t \cP_\alpha(t-\tau)f(\tau) \diff \tau
\end{align}
where 
\begin{equation}
	\label{eq:P_alpha}
	\cP_\alpha(t) 
	\defeq P_\alpha(t;\Ls)
	=t^{\alpha-1} E_{\alpha,\alpha}(-t^\alpha \Ls).
\end{equation}
This formula is already presented in \cite{Gal2020}. We emphasize that $\Semi_\alpha(t)$ and $\cP_\alpha(t)$ do not satisfy the semigroup property in general.

\paragraph{$L^2$ theory}

As usual, we have that
\begin{align*}
	\Big\| \mathcal H_C [u_0,0,0] (t) \Big\|_{L^2 (\Omega)}^2 
 &= \sum_{j=1}^\infty E_\alpha(-\lambda_j t^\alpha)^2 \langle u_0 , \varphi_j \rangle ^2 
 \le E_\alpha(-\lambda_1 t^\alpha)^2 \sum_{j=1}^\infty \angles{u_0,\varphi_j}^2 \\
	&= E_\alpha(-\lambda_1 t^\alpha)^2 \|u_0\|_{L^2 (\Omega)}^2.
\end{align*}
Therefore $S_\alpha : L^2 (\Omega) \to C(0,\infty; L^2 (\Omega))$ and we have the estimate
\begin{equation}
	\Big\| \mathcal H_C [u_0,0,0] (t) \Big\|_{L^2 (\Omega)} \le E_\alpha(-\lambda_1 t^\alpha) \|u_0\|_{L^2 (\Omega)}
\end{equation}
\begin{remark}
    Notice that, due to the slow decay of $E_\alpha$ we have that
    $
        {E_\alpha(-\lambda_n t^\alpha)} / {E_\alpha(-\lambda_1 t^\alpha)}
    $
    does not converge to zero as $t \to \infty$ for $n > 1$. Hence, unlike in the case $\alpha = 1$ we cannot simplify $\mathcal H_C [u_0,0,0] (t)$ to $E_\alpha(-\lambda_1 t^\alpha) \langle u_0, \varphi_1 \rangle \varphi_1$.
\end{remark}

Similarly, we have an $L^2$ estimate for finite time $t\in[0,T]$,
\begin{align*}
	\Big\| \mathcal H_C [0,f,0] (t) \Big\|_{L^2 (\Omega)}^2 
	&\le \left( \int_0^t \tau^{\alpha-1} E_{\alpha,\alpha} (-\lambda_1 \tau^\alpha) \diff \tau \right) \sup_{\tau \in [0,T]} \| f(\tau) \|_{L^2(\Omega)}^2\\
	&=\frac{1-E(-\lambda_1 t^\alpha)}{\lambda_1}
	\sup_{\tau \in [0,T]} \| f(\tau) \|_{L^2(\Omega)}^2.
\end{align*}

The solution for $f(t,x) = f(x)$ is cleanly expressed as
\begin{align*}
	\mathcal H_C [0,f,0] (t) 
	&= \sum_{j=1}^\infty \int_0^\infty \tau^{\alpha-1} E_{\alpha,\alpha} (-\lambda_j \tau^\alpha) \diff \tau 
	\,\langle f, \varphi_j \rangle \varphi_j 
	= \sum_{j=1}^\infty \frac{1-E_\alpha(-\lambda_j t^\alpha)}{\lambda_j} \langle f, \varphi_j \rangle \varphi_j.
\end{align*}
Hence its asymptotic behaviour is simply given by
\begin{align*}
	\lim_{t\to\infty}
	\mathcal H_C [0,f,0] (t) 
	&=
	\sum_{j=1}^\infty \frac{1}{\lambda_j} \langle f, \varphi_j \rangle \varphi_j
	=\Ls^{-1} f,
\end{align*}
i.e. the solution of the $\Ls$-Poisson equation. 
\begin{remark}
	Notice that the notation $\Ls^{-1}$ for the %
    Green operator (and more generally $(\Ls+\mu)^{-1}$ for the resolvent) is consistent with the notation \eqref{eq:FL}.
\end{remark}

\paragraph{Properties of the kernel representation}

In \cite[Section 2.1]{Gal2020} it is shown that, as $t \searrow 0$ we have
\begin{equation}
	\label{eq:continuity at 0} 
\begin{aligned} 
	\left\| \Semi_\alpha(t) [ v] - v \right\|_{L^2(\Omega)} &\longrightarrow 0 \\
	\left\| \tfrac{t^{\alpha - 1}}{\Gamma(\alpha)} \cP_\alpha(t) [v] - v \right\|_{L^2(\Omega)} &\longrightarrow 0.
\end{aligned}
\end{equation}   
We refer to \cite[Section 2.2]{Gal2020} for $L^p$-$L^q$ estimates which are recovered in terms of the contractivity of $\Semi(t)=e^{-t \Ls}$. We stress that $\Semi_\alpha(t):L^p(\Omega) \to L^q(\Omega)$ only when $\frac{n}{2s}(\frac1p-\frac1q)<1$ and $\cP_\alpha(t):L^p(\Omega) \to L^q(\Omega)$ only when $\frac{n}{2s}(\frac1p-\frac1q)<2$.

\normalcolor

	\subsection{Riemann--Liouville}
Similarly, since we have already defined 
\eqref{eq:P_alpha},
going back to 
\eqref{eq:ODE Riemann solution}
we observe that
\begin{align}
\label{eq:HR}
	\mathcal H_R [u_0,f,0] (t)
	&=\cP_\alpha(t)u_0
	+\int_0^t \cP_\alpha(t-\tau)f(\tau) \diff \tau.
\end{align}

Since $P_\alpha (t; \lambda) \sim \frac{t^{\alpha - 1}}{\Gamma(\alpha)}$ at $0$ (see \Cref{rem:Palpha singular at 0}), we cannot expect to construct a theory with $\cP_\alpha: L^2 (\Omega) \to L^\infty (0,T; L^2 (\Omega))$. Nevertheless, we do have $\cP_\alpha:L^2(\Omega) \to t^{\alpha-1} L^\infty(0,T;L^2(\Omega))$ with the corrected estimate
\begin{equation}
	\Big\| \mathcal H_R [u_0,0,0] (t) \Big\|_{L^2 (\Omega)} \le t^{\alpha - 1} E_{\alpha,\alpha} (-\lambda_1 t^\alpha) \| u_0 \|_{L^2}. 
\end{equation}

\begin{remark}
	\label{rem:equality u_0 = 0}
	We point out that
	$
	\mathcal H_R[0,f,0] (t) = \mathcal H_C[0,f,0] (t)
	$. Therefore, we define simply
	\begin{equation*}
		\mathcal H [0,f,0] (t) \defeq  \mathcal H_R [0,f,0] (t).
	\end{equation*}
\end{remark}
	
	\subsection{Well-posedness. Proof of \texorpdfstring{\Cref{thm:h = 0 L2}}{Theorem \ref{thm:h = 0 L2}}}
Uniqueness of spectral solutions of either initial value problem that lie in the spaces in the statement follows directly from the theory of fractional ODEs developed. By construction, our candidate solutions \eqref{eq:H representation} satisfy the spectral equation. 

The only missing detail, then, is the regularity of our candidate solutions. We have already proven that $\Semi_\alpha(t)[v], \cP_\alpha(t)[v] \in L^2 (\Omega)$. In fact, due to the inverse linear (respectively quadratic) decay of the Mittag-Leffler functions stated in \eqref{eq:ML-bound}, they are also in $\mathrm H(\Omega)$.

The continuity at $t = 0$ follows from \eqref{eq:continuity at 0}. Due to the bounds presented before, in fact $\Semi_\alpha(\cdot)[u_0], \cP_\alpha (t) u_0 \in C([\varepsilon,T]; \mathrm H(\Omega))$. The integral part is even easier.

The time differentiability follows from the explicit computation of $\frac{du}{dt}$ as in \cite[Proposition 2.1.9]{Gal2020}. We point out that $\Semi_\alpha' = \cP_\alpha \Ls$.
This can be done similarly for the Riemann--Liouville derivative.
\qed 

\begin{remark}
In \cite{Gal2020} the authors deal with the notion of \emph{strong solution}. This is also possible in our setting, but our interest in the very weak solutions described below.
\end{remark} 

\subsection{Very weak formulation when \texorpdfstring{$h = 0$}{h = 0}}
\label{sec:very weak formulation}
Due to \eqref{eq:ODE adjoint}, for every $T > 0$, $u_0, v_0 \in L^2 (\Omega)$ and $f,g \in L^2((0,T) \times \Omega)$ that
\begin{equation}
	\label{eq:PDE adjoint}
	\begin{aligned} 
	\int_0^T \int_\Omega &\mathcal H_R [v_0,g,0](T-t,x) f(t,x) \diff x \diff t + \int_\Omega u_0 (x)  \int_0^T \frac{ t^{-\alpha} }{\Gamma(1-\alpha)} \mathcal H_R[v_0,g,0] (T-t,x) \diff t \diff x \\
	& = \int_0^T \int_\Omega \mathcal H_C [u_0, f, 0](t , x) g(T-t,x) \diff x \diff t + \int_\Omega \Heat_C[u_0,f,0](T,x) v_0 (x) \diff x.
	\end{aligned} 
\end{equation}

This allows for a very natural definition of very weak solution:
This notion of solution yields uniqueness and positivity in a very standard way. It is also compatible with the $L^2$ theory constructed before. 

Integrability properties can be directly recovered from the estimates of the kernels of $\Semi_\alpha(t)$, $\cP_\alpha(t)$ that are directly related to those of $\Semi(t)=e^{-t\Ls}$.

\begin{remark}
    Notice that we could equivalently write that for every
	$v_0 \in L_c^\infty (\Omega)$ and for a.e. $t > 0$ we have
	\begin{equation}
		\label{eq:Caputo very weak}
		\begin{aligned} 
			\int_\Omega u(t,x) v_0 (x) \diff x &= 	\int_0^t \int_\Omega f(\sigma,x) \cP_\alpha[v_0](t-\sigma,x)  \diff \sigma \diff x \\
			&\qquad + \int_\Omega u_0 (x)  \int_0^t \frac{ \sigma^{-\alpha}}{\Gamma(1-\alpha)} \cP_\alpha[v_0](t-\sigma,x) \diff \sigma \diff x.
		\end{aligned} 
	\end{equation}
	This formulation is nicer for the $L^\infty$ estimates in time.
\end{remark} 

\section{Time-fractional problem when \texorpdfstring{$h = 0$}{h=0} beyond \texorpdfstring{$L^2$}{L2}}
\label{sec:h=0 L1 and inf}

	\subsection{Weighted \texorpdfstring{$L^1$}{L1} and \texorpdfstring{$L^\infty$}{Linf} theory}

When we leave the $L^2$ framework, we need to look beyond simple properties of $E_\alpha$ and $E_{\alpha,\alpha}$. 
It is here where the Mainardi function comes into play.

For example, \cite[Proposition 2.1.3]{Gal2020} uses the representation \eqref{eq:Salpha from S}, \eqref{eq:Palpha from S}, and properties of the Mainardi function to show that 
$$
	\| \Semi(t) u_0 \|_Y \le M \|u_0 \|_Y
$$ 
for all $u_0 \in Y$ implies
$$
	\| \Semi_\alpha(t) u_0 \|_Y \le C \| u_0 \|_Y \qquad \text{ and } \qquad t^{1-\alpha} \| \cP_\alpha(t) u_0 \|_Y \le C \|  u_0 \|_Y
$$
Therefore, for suitably integrable $f$, similar properties hold for $\mathcal H [0,f,0]$. 
In particular, the mass contractivity $\| \Semi(t) u_0 \|_{L^1(\Omega)} \le \|u_0\|_{L^1(\Omega)}$ allows us to construct an $L^1 (\Omega)$ theory.

\bigskip 

Furthermore, it is common that the first eigenfunction $\varphi_1(x)$ satisfies the boundary condition with a rate $\delta(x)^\gamma$ for some $\gamma$ positive.  This is the case, for example with the Restricted Fractional Laplacian ($\gamma = s$), Censored Fractional Laplacian ($\gamma = 2s - 1$ which is only defined for $s > \frac 1 2$), and the Spectral Fractional Laplacian ($\gamma = 1$). 
This is the expected boundary behaviour of all solution with good data, as we proved in \cite{AbatangeloGomezCastroVazquez2019} for the elliptic case and \cite{Chan+GC+Vazquez2020parabolic} for the parabolic case. In those papers, conditions are set on the Green kernel. 
However, it is more convenient for us now to set condition on the heat kernel. We set ourselves in a framework that covers the three main settings, where sharp estimates for the kernels are provided in \Cref{sec:heat kernel examples}.

\bigskip 

The canonical framework is that for good data we expect solutions in $\delta^\gamma L^\infty (\Omega)$ (a weighted space cointaining $\varphi_j$). The worst admissible data is in $L^1 (\Omega , \delta^\gamma)$, a fact guaranteed by the lower estimate
\begin{equation*}
	\Ls^{-1} [f] (x) \ge c_1 \delta(x)^\gamma \int_\Omega f(y) \delta(y)^\gamma \diff y, \qquad \forall f \ge 0.
\end{equation*}
where $c_1 > 0$.

\bigskip 

In general, under \eqref{eq:Green estimates}, we have that 
\begin{equation*}
	\Ls^{-1} [f] \asymp \delta^\gamma, \qquad \forall 0 \le f \in L^\infty_c (\Omega).
\end{equation*}

\begin{remark} 
In the local-in-time setting in \cite{Chan+GC+Vazquez2020parabolic} we showed the nice regularisation 
$$\Semi(t) : L^1(\Omega, \delta^\gamma) \to \delta^\gamma L^\infty (\Omega)$$ 
using the semigroup property.
Since we were not interested in the operator norm, conditions on the Green kernel sufficed. 
Due to the memory coming from the non-locality in time, we cannot expect such regularisation. Going back to \eqref{eq:Salpha from S} we have that
\begin{equation*}
	\frac{ \Semi_\alpha(t)[u_0] (x) }{ \delta(x) ^{\beta_1} } = \int_\Omega \left( \int_{0}^{\infty} \Phi_\alpha(\tau) \frac{\p ({\tau t^{\alpha}} ,x,y ) }{\delta(x)^{\beta_1} \delta(y)^{\beta_2} } \diff \tau \right)  u_0(y) \delta(y)^{\beta_2}  \diff y
\end{equation*}
Therefore, the regularisation relies on the integrals
\begin{equation*}
	\int_{0}^{\infty} \Phi_\alpha(\tau) \frac{  \p ({\tau t^{\alpha}} ,x,y ) }{\delta(x) ^{\beta_1} \delta(y)^{\beta_2} } \diff \tau.
\end{equation*}
Unfortunately, obtaining sharp estimates for such integral appears to be a non-trivial task.
\end{remark}

We start developing the theory of very weak solutions with a compactness estimate.

\begin{lemma}[Uniform space-time integrability in $L^1(0,T;L^1(\Omega,\delta^\gamma))$]
\label{lem:compactness}
Let $0\leq t_0<t_1\leq T$ and $A\subset \Omega$.  
\normalcolor
Then
\begin{align*}
\int_{t_0}^{t_1}&\int_{A}
    |\Heat_C[u_0,f,0]|
    \delta(x)^\gamma
\diff x \diff t \\
&\leq
    \omega_T(t_1-t_0)\omega(|A|)
    \left(
        \int_{\Omega}
        |u_0(x)|\delta^\gamma
        \diff x
        +\int_{0}^{T}\int_{\Omega}
            |f(t,x)|\delta(x)^\gamma
        \diff x \diff t
    \right).
\end{align*}
\end{lemma}

Here and below $\omega$ represents a modulus of continuity, i.e. a non-decreasing, non-negative function such that $\omega(0^+) = 0$. We denote the dependence by sub-indexes. 

\begin{proof}
By splitting into positive and negative parts, we may assume that $u_0,f,u\geq 0$.

\medskip

{\noindent\bf Step 1: Time compactness.} Take $\phi(t,x)=\chi_{[t_0,t_1]}(t)\varphi_1(x)$, so that
\begin{align*}
\Heat[0,\phi,0](t,x)
&=\int_{0}^{t}
    P_\alpha(t-\tau;\lambda_1)
    \phi(\tau,x)
\diff \tau
=\int_{(t-t_1)_+}^{(t-t_0)_+}
    -\frac{1}{\lambda_1}
    \frac{d}{d\tau}E_\alpha(-\lambda_1 \tau^\alpha)
\diff \tau
\,\varphi_1(x)\\
&=\frac{E_\alpha(-\lambda_1(t-t_1)_+^\alpha)-E_\alpha(-\lambda_1(t-t_0)_+^\alpha))}{\lambda_1}
\varphi_1(x)
=\omega_T(t_1-t_0)\varphi_1(x).
\end{align*}
Hence,
\begin{align*}
\int_{t_0}^{t_1}\int_{\Omega}
    |u(t,x)|\delta(x)^\gamma
\diff x \diff t
\leq
    \omega_T(t_1-t_0)
    \left(
        \int_{\Omega}
        |u_0(x)|\delta^\gamma
        \diff x
        +\int_{0}^{T}\int_{\Omega}
            |f(t,x)|\delta(x)^\gamma
        \diff x \diff t
    \right).
\end{align*}

\medskip

{\noindent \bf Step 2: Space compactness.} Take $\phi(t,x)=\chi_A(x)\varphi_1(x)$, such that
\begin{align*}
\Heat[0,\phi,0](t,x)
&=\int_{0}^{t}
    P_\alpha(t-\tau;\Ls)[\phi](\tau,x)
\diff \tau
=\int_{0}^{t}
    P_\alpha(\tau)
\diff\tau
\,[\chi_A\varphi_1](x)\\
&\leq
\int_{0}^{\infty}
    P_\alpha(\tau)
\diff \tau
\,[\chi_A\varphi_1](x)
=\int_{0}^{\infty}
\left(
    \alpha \tau^{\alpha-1}
    \int_{0}^{\infty}
        \sigma
        \Phi_\alpha(\sigma)
        \Semi(\sigma \tau^\alpha)
    \diff \sigma
\right)
\diff \tau
\,[\chi_A\varphi_1](x)\\
&=\int_{0}^{\infty}
\left(
    \int_{0}^{\infty}
        \Semi(\sigma\tau^\alpha)
    \diff(\sigma\tau^\alpha)
\right)
    \Phi_\alpha(\sigma)
\diff\sigma
\,[\chi_A\varphi_1](x)
=\Green[\chi_A\varphi_1](x)\\
&\leq \omega(|A|)\delta(x)^\gamma,
\end{align*}
where the last estimate follows from \eqref{eq:Green AGV} and the argument in \cite[Lemma 7.3]{Chan+GC+Vazquez2020parabolic}. Consequently,
\begin{align*}
\int_{0}^{T}\int_{A}
    |u(t,x)|\delta(x)^\gamma
\diff x \diff t
\leq
    \omega(|A|)
    \left(
        \int_{\Omega}
        |u_0(x)|\delta(x)^\gamma
        \diff x
        +\int_{0}^{T}\int_{\Omega}
            |f(t,x)|\delta(x)^\gamma
        \diff x \diff t
    \right).
\end{align*}

\medskip

{\noindent \bf Step 3: Space-time compactness.} Using {\bf Step 1--Step 2}, we have
\begin{align*}
    \int_{t_0}^{t_1}\int_{A}
    &    |u(t,x)|\delta(x)^\gamma
    \diff x \diff t
    \leq
    \left(
    \int_{t_0}^{t_1}\int_{\Omega}
        |u(t,x)|\delta(x)^\gamma
    \diff x \diff t
    \right)^{\frac12}
    \left(
    \int_{0}^{T}\int_{A}
        |u(t,x)|\delta(x)^\gamma
    \diff x \diff t
    \right)^{\frac12}
    \\
    &\leq
        \omega_T(t_1-t_0)
        \omega(|A|)
        \left(
            \int_{\Omega}
            |u_0(x)|\delta(x)^\gamma
            \diff x
            +\int_{0}^{T}\int_{\Omega}
                |f(t,x)|\delta(x)^\gamma
            \diff x \diff t
        \right).
\end{align*}

\medskip

{\noindent \bf Step 4: Space-time compactness for signed data}
In general we split $u_0=(u_0)_+-(u_0)_-$ and $f=f_+-f_-$, which yields
\[
u=\Heat[(u_0)_+,f_+,0]-\Heat[(u_0)_-,f_-,0].
\]
Then {\bf Step 3} can be applied to each summand, completing the proof.
\end{proof}

To obtain $L^1_{\rm loc}(\Omega)$ compactness, we apply the above estimate to each $K\Subset \Omega$.

\begin{remark}
    Notice that the only crucial ingredients in the proof above is that $\varphi_1 \asymp \delta^\gamma$ and $\Green[\chi_A \varphi_1] \le \omega(|A|) \varphi_1$, which are minimal assumptions on $\Green$ that uses only mild integrability assumptions, not the exact shape. In Lipschitz domains, where it can happen that $\varphi_1 \not \asymp \delta^\gamma$ for any $\gamma$, the correct weight is $\varphi_1$.
\end{remark}

	\subsection{Well-posedness. Proof of \texorpdfstring{\Cref{thm:h = 0}}{Theorem \ref{thm:h = 0}}}

When $u_0$ and $f$ are regular, we have proven that \eqref{eq:H representation} is a spectral solution. As described in \Cref{sec:very weak formulation}, this solution is a weak solution. 

Let $u_0, f$ be in the general classes of the statement. They can be approximated by $u_{0k}, f_k$ smooth. 
Because of the \emph{a priori} estimates proven, $\Heat[u_{0k},f_k,0] \to \Heat[u_0,f,0]$ in $ L^1(0,T;  L^1(\Omega,\allowbreak \delta^\gamma))$. 
Due to the regularity of $\Heat[0,\phi,0]$ we can pass to the limit in the definition of weak-dual solution. This guarantees existence.

Finally, we prove the uniqueness. Assume there are two weak-dual solutions. Let $w$ be their difference. Since they share a right-hand side in \Cref{def:weak-dual} we recover, for each $T$ and $\phi$ smooth
$$
    \int_0^T \int_\Omega w(t,x) \phi(T-t,x) \diff x \diff t = 0.
$$
For any $K \Subset \Omega$, taking
$$
    \phi(t,x) =   \chi_K  (x) \sign w(T-t,x),
$$
we conclude that $w = 0$ a.e. in $[0,T] \times \Omega$. This completes the proof. 
\qed

\subsection{Sharp boundary behaviour for good data}

We derive estimates for
\begin{equation*}%
\bP_\alpha(t,x,y)
=\alpha t^{\alpha-1}
    \int_0^\infty
        \tau\Phi_\alpha(\tau)
        \bS(\tau t^\alpha,x,y)
    \diff\tau,
\qquad t>0,\,x,y\in\Omega,
\end{equation*}
using the corresponding estimates of the heat kernel $\bS$.

\begin{remark}
    We point out the pointwise estimate
    \begin{align*}
        \mathbb P_\alpha (t,x,y) 
        &=  
        \alpha t^{\alpha-1} 
        \int_0^\infty (\sigma t^{-\alpha} ) \Phi( \sigma t^{-\alpha} ) \p (\sigma ,x,y) \diff \sigma t^{-\alpha} \\
        &=  
        \alpha t^{-1} 
        \int_0^\infty (\sigma t^{-\alpha} ) \Phi( \sigma t^{-\alpha} ) \p (\sigma ,x,y) \diff \sigma  \\
        &\le  
        \alpha t^{-1} \| \tau \Phi(\tau) \|_{L^\infty} 
        \G(x,y).
    \end{align*}
    Unfortunately, in this direct  computation one loses some power of $t$ and the integrability in time. Alternatives, such as (weighted) integral estimates, will be used to fit our purposes.
\end{remark}

\begin{lemma}
\label{lem:D Pa 1}
We have that $\bP_\alpha(\cdot,x,\cdot)/\delta(x)^\gamma \in L^1(0,\infty;L^1(\Omega,\delta^\gamma))$, uniformly in $x\in\Omega$.
\end{lemma}

\begin{proof}
We compute
\begin{align*}
\int_0^\infty\int_\Omega
    \frac{
        \bP_\alpha(t,x,y)
    }{\delta(x)^\gamma}
    \delta(y)^\gamma
\diff y\diff t
&=\int_0^\infty\int_\Omega
    \alpha t^{\alpha-1}
    \int_0^\infty
        \tau\Phi_\alpha(\tau)
        \frac{
            \bS(\tau t^\alpha,x,y)
        }{
            \delta(x)^\gamma
        }
    \diff\tau
\cdot\delta(y)^\gamma
\diff y\diff t\\
&=\int_0^\infty\int_\Omega
\int_0^\infty
    \Phi_\alpha(\tau)
    \frac{
        \bS(\tau t^\alpha,x,y)
    }{
        \delta(x)^\gamma
    }
    \alpha\tau t^{\alpha-1}
\diff t
\cdot\delta(y)^\gamma
\diff y\diff\tau\\
&=\int_0^\infty\int_\Omega
\int_0^\infty
    \Phi_\alpha(\tau)
    \frac{
        \bS(\sigma,x,y)
    }{
        \delta(x)^\gamma
    }
\diff\sigma
\cdot\delta(y)^\gamma
\diff y\diff\tau\\
&=\int_\Omega
    \frac{\bG(x,y)}{\delta(x)^\gamma}
    \delta(y)^\gamma
\diff y.
\end{align*}
Using \eqref{eq:Green AGV}, the latter integral is bounded uniformly for $x\in\Omega$, as desired.
\end{proof}

In \cite{Chan+GC+Vazquez2020parabolic} we proved that $\Semi(t) : \mathcal M (\Omega,\delta^\gamma) \to \delta^\gamma C(\overline \Omega)$ is a continuous operator for any $t > 0$. By the continuity of the linear operator $\Semi(t)$ and the semigroup property, we deduce that
\begin{equation*}
	\frac{ \p  (t,x,y)}{\delta(x)^\gamma} = \frac{\Semi(t)[\delta_y] (x)}{\delta(x)^\gamma }\in C((0,T] \times \overline \Omega \times \Omega).
\end{equation*}
Here $\delta_y$ is the Dirac delta distribution centred at $y$, whereas $\delta$ is the distance function. 
Due to this continuity, the following limit is well defined:
\begin{equation}
\label{eq:D S}
	D_\gamma \p  (t, \zeta , y) = \lim_{x \to \zeta}  \frac{ \p  (t,x,y)}{\delta(x)^\gamma}.
\end{equation}

\begin{lemma}
\label{lem:D Pa 2}
For any $\zeta\in\partial\Omega$,
\[
D_\gamma\bP_\alpha(\cdot,\zeta,\cdot)
:=\lim_{\Omega\ni x\to\zeta}
\frac{\bP_\alpha(\cdot,x,\cdot)}{\delta(x)^\gamma}
\]
exists in $L^1(0,\infty;L^1(\Omega,\delta^\gamma))$ and is equal to
\[
D_\gamma\bP_\alpha(t,\zeta,y)
=\alpha t^{\alpha-1}
\int_0^\infty
    \tau\Phi_\alpha(\tau)
    D_\gamma\bS(\tau t^\alpha,\zeta,y)
\diff\tau,
\qquad t>0,\,\zeta\in\partial\Omega,\,y\in\Omega.
\]
\end{lemma}

\begin{proof}
We estimate
\begin{align*}
&\quad\;
\int_0^\infty\int_\Omega
\bigg|
    \frac{
        \bP_\alpha(t,x,y)
    }{\delta(x)^\gamma}
    -\alpha t^{\alpha-1}
    \int_0^\infty
        \tau\Phi_\alpha(\tau)
        D_\gamma\bS(\tau t^\alpha,\zeta,y)
    \diff\tau
\bigg|
    \delta(y)^\gamma
\diff y\diff t\\
&\leq\int_0^\infty\int_\Omega
    \alpha t^{\alpha-1}
    \int_0^\infty
        \tau\Phi_\alpha(\tau)
        \bigg|
        \frac{
            \bS(\tau t^\alpha,x,y)
        }{
            \delta(x)^\gamma
        }
        -D_\gamma\bS(\tau t^\alpha,\zeta,y)
        \bigg|
    \diff\tau
\cdot\delta(y)^\gamma
\diff y\diff t\\
&=\int_0^\infty\int_\Omega
\int_0^\infty
    \Phi_\alpha(\tau)
    \bigg|
    \frac{
        \bS(\tau t^\alpha,x,y)
    }{
        \delta(x)^\gamma
    }
    -D_\gamma\bS(\tau t^\alpha,\zeta,y)
    \bigg|
    \alpha\tau t^{\alpha-1}
\diff t
\cdot\delta(y)^\gamma
\diff y\diff\tau\\
&=\int_0^\infty\int_\Omega
\int_0^\infty
    \Phi_\alpha(\tau)
    \bigg|
    \frac{
        \bS(\sigma,x,y)
    }{
        \delta(x)^\gamma
    }
    -D_\gamma\bS(\sigma,\zeta,y)
    \bigg|
\diff\sigma
\cdot\delta(y)^\gamma
\diff y\diff\tau\\
&=\int_\Omega
    \int_0^\infty
    \bigg|
    \frac{
        \bS(\sigma,x,y)
    }{
        \delta(x)^\gamma
    }
    -D_\gamma\bS(\sigma,\zeta,y)
    \bigg|
    \diff\sigma
    \cdot
    \delta(y)^\gamma
\diff y.
\end{align*}
By \cite[Theorem 6.1]{Chan+GC+Vazquez2020parabolic}, the quotient $\bS(\sigma,x,y)/\delta(x)^\gamma$ is continuous in $x$ up to the boundary for each $(\sigma,y)\in(0,\infty)\times \Omega$. In particular, the last integrand tends to $0$ a.e. Moreover, this integrand is dominated by $    \frac{
        \bS(\sigma,x,y)
    }{
        \delta(x)^\gamma
    }
    +D_\gamma\bS(\sigma,\zeta,y)
$ which is integrable in $L^1(0,\infty;L^1(\Omega,\delta^\gamma))$ as shown in \Cref{lem:D Pa 1}. By Dominated Convergence Theorem, the last integral tends to zero as $x \to \zeta$. %
\end{proof}

Now we prove that if $\phi \in \delta^\gamma L^\infty ( (0,T)\times \Omega ) $, then so is $\Heat[0,\phi,0]$, and in addition $D_\gamma \Heat[0,\phi,0]$ exists.

\begin{lemma}
For any $\phi\in \delta^\gamma L^\infty((0,T)\times\Omega)$,
\[
\Heat[0,\phi,0](t,x)
=\int_0^t\int_\Omega
    \bP_\alpha(t-\tau,x,y)
    \phi(\tau,y)
\diff y \diff \tau
\]
lies also in $\delta^\gamma L^\infty((0,T)\times\Omega)$.
\end{lemma}

\begin{proof}
We estimate
\begin{align*}
\frac{\Heat[0,\phi,0](t,x)}{\delta(x)^\gamma}
&=\int_0^t\int_\Omega
    \frac{\bP_\alpha(t-\tau,x,y)}{\delta(x)^\gamma}
    \phi(\tau,y)
\diff y \diff \tau\\
&\leq
\norm[L^\infty((0,T)\times\Omega)]{
    \frac{\phi}{\delta^\gamma}
}
\int_0^t\int_\Omega
    \frac{\bP_\alpha(t-\tau,x,y)}{\delta(x)^\gamma}
    \delta(y)^\gamma
\diff y \diff \tau\\
&\leq
\norm[L^\infty((0,T)\times\Omega)]{
    \frac{\phi}{\delta^\gamma}
}
\int_0^\infty\int_\Omega
    \frac{\bP_\alpha(\tau,x,y)}{\delta(x)^\gamma}
    \delta(y)^\gamma
\diff y \diff \tau.
\end{align*}
By \Cref{lem:D Pa 1}, the last double integral is bounded by a uniform constant, as desired.
\end{proof}

\begin{lemma}
For any $\phi\in \delta^\gamma L^\infty((0,T)\times\Omega)$, 
\[
D_\gamma\Heat[0,\phi,0](t,\zeta)
:=\lim_{\Omega\ni x\to\zeta}
\frac{\Heat[0,\phi,0](t,x)}{\delta(x)^\gamma}
\]
exists in $L^\infty((0,T)\times\partial\Omega)$ and is equal to
\[
D_\gamma\Heat[0,\phi,0](t,\zeta)
=\int_0^t\int_\Omega
    D_\gamma\bP_\alpha(t-\tau,\zeta,y)
    \phi(\tau,y)
\diff y \diff \tau.
\]
\end{lemma}

\begin{proof}
We estimate
\begin{align*}
&\quad\;
\bigg|
\frac{\Heat[0,\phi,0](t,x)}{\delta(x)^\gamma}
-\int_0^t\int_\Omega
    D_\gamma\bP_\alpha(t-\tau,\zeta,y)
    \phi(\tau,y)
\diff y \diff \tau
\bigg|\\
&\leq
\norm[L^\infty((0,T)\times\Omega)]{
    \frac{\phi}{\delta^\gamma}
}
\int_0^t\int_\Omega
\bigg|
    \frac{\bP_\alpha(t-\tau,x,y)}{\delta(x)^\gamma}
    -D_\gamma\bP_\alpha(t-\tau,\zeta,y)
\bigg|
\,\delta(y)^\gamma
\diff y \diff \tau\\
&\leq
\norm[L^\infty((0,T)\times\Omega)]{
    \frac{\phi}{\delta^\gamma}
}
\int_0^\infty\int_\Omega
\bigg|
    \frac{\bP_\alpha(\tau,x,y)}{\delta(x)^\gamma}
    -D_\gamma\bP_\alpha(\tau,\zeta,y)
\bigg|
\,\delta(y)^\gamma
\diff y \diff \tau\\
\end{align*}
By \Cref{lem:D Pa 2}, the last integral %
converges to $0$ as $x\to\zeta$. %
\end{proof}

\section{Singular boundary condition when \texorpdfstring{$h \ne 0$}{h!=0}}
\label{sec:h ne 0}

\subsection{Concentration of \texorpdfstring{$f$}{f} towards singular boundary data}

\begin{definition} 
    We define a very weak solution for $u_0 = 0, f = 0, $ and $h \ne 0$ as a function $u \in L^1 (0,T; L^1(\Omega,\delta^\gamma))$ which satisfies
    \begin{align}
    \label{eq:very weak h ne 0}
        \int_0^T \int_\Omega u(t,x) \phi(T-t,x) \diff x\diff t
        &= \int_0^T \int_{\partial \Omega} h(t, \zeta) D_\gamma \mathcal H[0,\phi,0] (T-t, \zeta) \diff \zeta \diff t,
    \end{align}
for any $\phi\in \delta^\gamma L^\infty((0,T)\times\Omega)$.
\end{definition}
Given this definition, uniqueness is trivial.

\begin{lemma}
\label{lem:unique}
Suppose $u\in L^1(0,T;L^1_{\rm loc}(\Omega))$ satisfies that for any $\phi\in L^\infty(0,T;L^\infty_c(\Omega))$,
\[
\int_0^T \int_\Omega u(t,x) \phi(T-t,x) \diff x\diff t=0,
\]
then $u\equiv 0$ in $\Omega$. In particular, the same implication holds for $u\in L^1(0,T;L^1(\Omega,\delta^\gamma))$ with test functions $\phi\in\delta^\gamma L^\infty((0,T)\times\Omega)$.
\end{lemma}

\begin{proof}
For every $K\Subset\Omega$, choosing $\phi(t,x)=\sign u(T-t,x) \chi_{K}(x)$  yields $\int_0^T\int_K |u| \diff x\diff t=0$.
\end{proof}

First we check that the solution lies in the correct weighted space.

\begin{lemma}
Given $h\in L^1((0,T)\times \partial\Omega)$, $u=\Heat[0,0,h]$ given by \eqref{eq:H representation} lies in $L^1(0,T;L^1(\Omega,\delta^\gamma))$. Moreover,
\[
\int_0^T\int_{\Omega}
    u(x,t)
    \delta(x)^\gamma
\diff x\diff t
\leq C\int_0^T\int_{\partial\Omega}
    h(\tau,\zeta)
\diff\zeta\diff\tau.
\]
\end{lemma}

\begin{proof}
We express
\begin{align*}
\int_0^T u(t,x)\diff t
&=\int_0^T\int_0^t\int_{\partial\Omega}
    D_\gamma\bP_\alpha(t-\tau,x,\zeta)
    h(\tau,\zeta)
\diff\zeta \diff\tau \diff t\\
&=\int_0^T\int_{\partial\Omega}
\bigg[
\int_\tau^T
    D_\gamma\bP_\alpha(t-\tau,x,\zeta)
\diff t
\bigg]
    \cdot h(\tau,\zeta)
\diff\zeta\diff\tau\\
&=\int_0^T\int_{\partial\Omega}
\bigg[
    \int_{0}^{T-\tau}
    D_\gamma\bP_\alpha(T-\tau-t,x,\zeta)
    \diff t
\bigg]
\cdot h(\tau,\zeta)
\diff\zeta\diff\tau\\
&=\int_0^T\int_{\partial\Omega}
\bigg[
    \int_{0}^{T-\tau}
    D_\gamma\bP_\alpha(t,x,\zeta)
    \diff t
\bigg]
\cdot h(\tau,\zeta)
\diff\zeta\diff\tau\\
&\leq
\int_{\partial\Omega}
    \bigg[\int_0^\infty
    D_\gamma\bP_\alpha(t,x,\zeta)
    \diff t\bigg]
    \bigg[\int_0^T
    h(\tau,\zeta)
    \diff\tau\bigg]
\diff\zeta.
\end{align*}
Using \Cref{lem:D Pa 2}, the last $t$-integral is in $L^1(\Omega,\delta^\gamma)$ (in variable $x$) and hence the result follows.
\end{proof}

Integrating by parts, we see that the only possible solution is precisely \eqref{eq:H representation}. 

\begin{lemma}
\label{lem:representation 00h}
Given $h\in L^1((0,T)\times \partial\Omega)$, $u=\Heat[0,0,h]$ given by \eqref{eq:H representation} satisfies \eqref{eq:very weak h ne 0}.
\end{lemma}

\begin{proof}
Keeping in mind that
\[\begin{split}
\Heat[0,\phi,0](t,x)
&=\bigg[\int_0^t \cP_\alpha(t-\tau)\phi(\tau) \diff\tau\bigg](x)
=\int_0^t\int_\Omega
    \bP_\alpha(t-\tau,x,y)
\phi(\tau,y) \diff y \diff \tau,
\end{split}\]
we verify that
\begin{align*}
&\quad\;
\int_0^T\int_\Omega
    \bigg[
    \int_0^t\int_{\partial\Omega}
        (D_\gamma\bP_\alpha)(t-\tau,x,\zeta)
        h(\tau,\zeta)
    \diff\zeta\diff\tau
    \bigg]
    \phi(T-t,x)
\diff x\diff t\\
&=\int_0^T\int_{\partial\Omega}
\bigg[
    \int_\tau^T\int_{\Omega}
        (D_\gamma\bP_\alpha)(t-\tau,x,\zeta)
        \phi(T-t,x)
    \diff x\diff t
\bigg]
h(\tau,\zeta)
\diff\zeta\diff\tau\\
&=\int_0^T\int_{\partial\Omega}
D_\gamma
\bigg[
    \int_{0}^{T-\tau}\int_{\Omega}
        \bP_\alpha(T-t-\tau,x,\cdot)
        \phi(t,x)
    \diff x\diff t
\bigg](\zeta)\,
h(\tau,\zeta)
\diff\zeta\diff\tau\\
&=\int_0^T\int_{\partial\Omega}
    D_\gamma\Heat[0,\phi,0](T-\tau,\zeta)
    \,h(\tau,\zeta)
\diff\zeta\diff\tau.\qedhere
\end{align*}
\end{proof}

\subsection{Well-posedness when 
    \texorpdfstring{$u_0, f = 0$ and $h \ne 0$}{u0, f = 0 and h != 0}. Proof of \texorpdfstring{\Cref{thm:h ne 0}}{Theorem \ref{thm:h ne 0}} i) }

The proof is structured in several steps, using the previous lemmas: 
\begin{enumerate}
		\item
		Due to compactness from \Cref{lem:compactness}, there is at least a convergent subsequence of $\mathcal H[0,\allowbreak f_j,0]$ in the sense $L^1 (0,T ; L^1 (\Omega, \delta^\gamma))$ to some function $u$.
		
		\item 
		By passing to the limit in \eqref{eq:very weak h = 0} we observe that $u$ satisfies \eqref{eq:very weak h ne 0}. 
		
		\item  Due to \Cref{lem:unique}, $u$ is the unique $L^1(0,T; L^1 (\Omega, \delta^\gamma))$ solution of \eqref{eq:very weak h ne 0}. By uniqueness of the weak limit, we deduce the convergence of the whole sequence $\mathcal H[0,f_j,0]$.
		
		\item Due to \Cref{lem:representation 00h}, $u$ is given precisely by \eqref{eq:H representation}. 
	\end{enumerate}
	This completes the proof.
\qed

	\subsection{\texorpdfstring{$\mathcal H[0,0,h]$}{H[0,0,h]} satisfies the singular boundary condition. Proof of \texorpdfstring{\Cref{thm:h ne 0}}{Theorem \ref{thm:h ne 0}} ii) }

Now we want to see whether \eqref{eq:singular boundary condition} holds.
We observe that \eqref{eq:H**} is a form of saying that $D_\gamma \mathbb S$ uniformly localises at $t = 0$ on the boundary.

\begin{theorem}
 Let $h \in C(\partial \Omega)$ and assume \eqref{eq:H**}, in addition to the main assumptions throughout the paper. Then \eqref{eq:singular boundary condition} holds.
\end{theorem}

\begin{proof} 
Using \eqref{eq:H representation}, we write down the ratio
\begin{align*}
	\frac{\mathcal H[0,0,h] (t,x) }{ u^\star (x) }  
	&=  
	\dfrac
	{ \int_0^t \int_{\partial \Omega} (D_\gamma \mathbb P_\alpha)(t - \tau,x,\zeta) h(\tau, \zeta) \diff \zeta \diff \tau }
	{  \int_{\partial \Omega} D_\gamma \G (\tilde \zeta,x) \diff \tilde \zeta   } \\
	&=
	\int_0^t \int_{\partial \Omega} 
	\frac
	{ (D_\gamma \mathbb P_\alpha)(t - \tau,x,\zeta) } 
	{  \int_{\partial \Omega} D_\gamma \G (\tilde \zeta,x) \diff \tilde \zeta   }
	h(\tau, \zeta) \diff \zeta \diff \tau 
\end{align*}
We define
\begin{equation*}
	\Upsilon(t,x,\zeta) = \frac
	{ (D_\gamma \mathbb P_\alpha)(t,x,\zeta) } 
	{  \int_{\partial \Omega} D_\gamma \G (\tilde \zeta,x) \diff \tilde \zeta   }.
\end{equation*}
Then, we notice that
\begin{equation} 
\label{eq:singular boundary Upsilon integral}
\begin{aligned}
	\int_0^t   \Upsilon(\sigma ,x,\zeta) \diff \sigma
	&= 
	\int_0^t \int_{0}^{\infty} 
	\alpha \sigma^{\alpha-1} \tau \Phi_\alpha(\tau)
	\dfrac
	{   D_\gamma  \p  (\tau \sigma^{\alpha},x,\zeta)  }
	{  \int_{\partial \Omega} \int_0^\infty D_\gamma \p (\tilde\tau, \tilde \zeta,x) \diff \tilde \zeta \diff \tilde\tau  } 
	\diff \tau
	\diff \sigma
	\\
	&= 	
	\int_0^\infty 
	\Phi_\alpha(\tau) 
	\left(
	\dfrac
	{  \int_{0}^{\tau t^\alpha}  D_\gamma  \p  (\tau \sigma^{\alpha},x,\zeta) 	\diff (\tau \sigma^\alpha)}
	{  \int_{\partial \Omega} \int_0^\infty D_\gamma \p (\tilde\tau, \tilde \zeta,x) \diff \tilde \zeta \diff \tilde\tau  } 
	\right) 
	\diff \tau
	\\
	&= 	
	\int_0^\infty
	\Phi_\alpha(\tau) 
	\dfrac
	{  \int_{0}^{\tau t^\alpha}    D_\gamma  \p  (\sigma,x,\zeta)  \diff \sigma }
	{  \int_0^\infty \int_{\partial \Omega}  D_\gamma \p (\tilde\tau, \tilde \zeta,x) \diff \tilde \zeta \diff \tilde\tau  } 
	\diff \tau
\end{aligned}
\end{equation}
We recover that
\begin{equation}
    \label{eq:integral of Upsilon}
    \forall x \in \Omega \text{ it holds that }	\int_0^\infty  \int_{\partial \Omega}  \Upsilon(\sigma ,x,\zeta) \diff \zeta \diff \sigma = 1.
\end{equation}
Notice that due to \eqref{eq:integral of Upsilon} and \eqref{eq:H**} we have that for any $t>0$,
\begin{equation}
    \label{eq:Upsilon is almost probability}
    \lim_{x \to \zeta_0} \int_0^t \int_{\partial \Omega} \Upsilon(\sigma, x,\zeta) \diff \zeta \diff \sigma = 1.
\end{equation}

We compute the following limit as $x \to \zeta_0$,
\begin{align*}
    \Bigg| \frac{\mathcal H[0,0,h] (t,x) }{ u^\star (x) } &- \left( \int_0^t \int_{\partial \Omega} \Upsilon(t-\sigma, x, \zeta) \diff \zeta \diff \sigma  \right)  h(t,\zeta_0) \Bigg| \\
    &\le \int_0^t \int_{\partial \Omega} \Upsilon(t-\sigma, x, \zeta)  |h(\sigma, \zeta) - h(t,\zeta_0)|   \diff \zeta   \diff \sigma.
\end{align*}
Assume that $h$ is continuous.
Now we split this into different parts:
\begin{enumerate}
    \item 
    Close to $(t, \zeta_0)$.
    We pick $\delta$ such that ball in $(\sigma, \zeta) \in  B(t, \delta) \times ( B(\zeta_0, \delta) \cap \partial \Omega )$ we have $|h(\sigma, \zeta) - h(t,\zeta_0)| \le \varepsilon$. Then
    \begin{equation*}
        \int_{t-\delta}^t \int_{B(\zeta_0, \delta) \cap \partial \Omega} \Upsilon(t-\sigma, x, \zeta)  |h(\sigma, \zeta) - h(t,\zeta_0)|   \diff \zeta   \diff \sigma \le \varepsilon  \int_0^\infty \int_{\partial \Omega} \Upsilon(\sigma, x , \zeta )\diff \zeta \diff \sigma = \varepsilon.
    \end{equation*}
    
    \item On $(0,t) \times ( \partial \Omega \setminus B(\zeta_0, \delta) ) $. In this region we use \eqref{eq:singular boundary Upsilon integral} to deduce that
    \begin{align*}
        \int_{0}^t \int_{\partial \Omega \setminus B(\zeta_0, \delta) } & \Upsilon(t-\sigma, x, \zeta)  |h(\sigma, \zeta) - h(t,\zeta_0)|   \diff \zeta   \diff \sigma  \\
        &\le 2 \|h\|_{L^\infty (\partial \Omega)}  
        \int_{\partial \Omega \setminus B(\zeta_0, \delta) } \int_0^\infty \Upsilon(\sigma, x, \zeta) \diff \sigma \diff \zeta \\
        &= 2 \|h\|_{L^\infty (\partial \Omega)}  
        \int_0^\infty
    	\Phi_\alpha(\tau) \diff \tau 
    	\dfrac
    	{  \int_{\partial \Omega \setminus B(\zeta_0, \delta) } \int_{0}^{\infty}    D_\gamma  \p  (\sigma,x,\zeta)  \diff \sigma \diff \zeta}
    	{  \int_0^\infty \int_{\partial \Omega}  D_\gamma \p (\tilde\tau, \tilde \zeta,x) \diff \tilde \zeta \diff \tilde\tau  } \\
    	&= 2 \|h\|_{L^\infty (\partial \Omega)}  
    	\dfrac
    	{  \int_{\partial \Omega \setminus B(\zeta_0, \delta) } D_\gamma \mathbb G(x,\zeta) \diff \zeta}
    	{   \int_{\partial \Omega}  D_\gamma \mathbb G(x, \tilde \zeta) \diff \tilde \zeta   } 
    \end{align*}
    because of the hypothesis we made above.
    Taking a smooth non-negative function $\varphi$ that takes value $1$ in $\partial \Omega \setminus B(\zeta_0, \delta)$ and $\varphi(\zeta_0) = 0$, we use that in \cite{AbatangeloGomezCastroVazquez2019} the authors prove 
    \begin{equation*}
    \lim_{x \to \zeta_0} \dfrac
	{  \int_{\partial \Omega } D_\gamma \mathbb G(x,\zeta) \varphi (\zeta) \diff \zeta}
	{   \int_{\partial \Omega}  D_\gamma \mathbb G(x, \tilde \zeta) \diff \tilde \zeta   } \to \varphi (\zeta_0) = 0.
    \end{equation*}
    
    \item Lastly, the region $(0,t-\delta) \times ( B(\zeta_0, \delta) \cap \partial \Omega)  $.
    \begin{align*}
         \int_{0}^{t-\delta}  \int_{\partial \Omega \setminus B(\zeta_0, \delta) } & \Upsilon(t-\sigma, x, \zeta)  |h(\sigma, \zeta) - h(t,\zeta_0)|   \diff \zeta   \diff \sigma  \\
         &\le 2 \|h \|_{L^\infty} \int_{0}^{t-\delta} \int_{\partial \Omega} \Upsilon(t-\sigma, x, \zeta) \diff \sigma \diff \zeta \\
         &\le 2 \|h \|_{L^\infty} \int_{\delta}^{t} \int_{\partial \Omega} \Upsilon(\sigma, x, \zeta) \diff \sigma \diff \zeta.
    \end{align*}    
    Now we notice that
    \begin{align*}
        \int_{\delta}^{t} &\int_{\partial \Omega} \Upsilon(\sigma, x, \zeta) \diff \sigma \diff \zeta  = \int_{0}^{t } \int_{\partial \Omega} \Upsilon(\sigma, x, \zeta) \diff \sigma \diff \zeta - \int_{0}^{\delta} \int_{\partial \Omega} \Upsilon(\sigma, x, \zeta) \diff \sigma \diff \zeta\\
        &=\int_0^\infty
	\Phi_\alpha(\tau) 
	\dfrac
	{  \int_{\tau \delta^\alpha}^{\tau t^\alpha} \int_{\partial \Omega}   D_\gamma  \p  (\sigma,x,\zeta) \diff \zeta \diff \sigma }
	{  \int_0^\infty \int_{\partial \Omega}  D_\gamma \p (\tilde\tau, \tilde \zeta,x) \diff \tilde \zeta \diff \tilde\tau  } 
	\diff \tau \diff \sigma \\
    &\le \int_0^\infty
	\Phi_\alpha(\tau) 
	\dfrac
	{  \int_{\tau \delta^\alpha}^{\infty} \int_{\partial \Omega}   D_\gamma  \p  (\sigma,x,\zeta) \diff \zeta \diff \sigma }
	{  \int_0^\infty \int_{\partial \Omega}  D_\gamma \p (\tilde\tau, \tilde \zeta,x) \diff \tilde \zeta \diff \tilde\tau  } 
	\diff \tau \diff \sigma 
    \end{align*} 
    As $x \to \zeta_0$ this converges to $0$ due to \eqref{eq:H**}.
\end{enumerate}
We have proved that, for any $\varepsilon > 0$ we have
\begin{equation*}
    \lim_{x \to \zeta_0} 
    \Bigg| \frac{\mathcal H[0,0,h] (t,x) }{ u^\star (x) } - \left( \int_0^t \int_{\partial \Omega} \Upsilon(t-\sigma, x, \zeta) \diff \zeta \diff \sigma  \right)  h(t,\zeta_0) \Bigg| \le \varepsilon .
\end{equation*}
Recalling \eqref{eq:Upsilon is almost probability} the proof is finished. 
\end{proof}

\appendix 
\section{Heat kernel estimates for the examples}
\label{sec:heat kernel examples}

For the heat kernel of $\partial_t u +\Ls u = f$ has the following estimates:
\begin{enumerate}
	\item For the heat kernel of the fractional Laplacian in the whole space known that
	$$
	p(t,x,y) \asymp \left(  t^{-\frac{d}{2s} }  \wedge \frac{t}{|x-y|^{d+2s}} \right)
	\asymp
	t^{-\frac{d}{2s}}
	\left(
	    1\wedge
	    \frac{t^{\frac{1}{2s}}}{|x-y|}
	\right)^{d+2s}.
	$$
    They can be recovered from computing the inverse Fourier transform of $e^{-t|\xi|^{2s}}$. It is particularly interesting to point out that for $s = \frac 12$ and $d = 1$ we have density function of Cauchy distribution
    $$
        p(t,x,y) = \frac 1 \pi \frac{t}{|x-y|^2 + t^2}.
    $$
	
	\item Restricted Fractional Laplacian with Dirichlet exterior condition (see \cite{ChenKimSong2010,Bogdan2010})
	
	$$ 
	\mathbb S_{RFL} (t,x,y) \asymp \left( 1 \wedge \frac{\delta(x)}{ t^{\frac{1}{2s}}} \right)^{s}
	\left( 1 \wedge \frac{\delta(y)}{t^{\frac{1}{2s}}} \right)^{s} p(t,x,y).
	$$
	
	\item Censored (or Regional) Fractional Laplacian with Dirichlet boundary condition for $s\in(\frac12,1)$ as can be found in \cite{Chen2009}
	
	$$ 
	\mathbb S_{CFL} (t,x,y) \asymp \left( 1 \wedge \frac{\delta(x)}{ t^{\frac{1}{2s}}} \right)^{2s-1}
	\left( 1 \wedge \frac{\delta(y)}{t^{\frac{1}{2s}}} \right)^{2s-1} p(t,x,y).
	$$
	
	\item For the spectral fractional Laplacian we have (see \cite{Song2004,Song_2020})
	
	$$
	\mathbb S_{SFL} (t,x,y) \asymp \left( 1 \wedge \frac{\delta(x)}{|x-y| + t^{\frac {1}{2s}} } \right)  \left( 1 \wedge \frac{\delta(y)}{|x-y| + t^{\frac {1}{2s}} } \right) p(t,x,y).
	$$

\end{enumerate}

\section*{Acknowledgements}
HC has received funding from the Swiss National Science Foundation under the Grant PZ00P2\_\-202012/1. DGC was partially supported by PID2021-127105NB-I00 from the Spanish Government and the Advanced Grant Nonlocal-CPD (Nonlocal PDEs for Complex Particle Dynamics: Phase Transitions, Patterns and Synchronization) of the European Research Council Executive Agency (ERC) under the European Union’s Horizon 2020 research and innovation programme (grant agreement No. 883363).
JLV was funded by grants PGC2018-098440-B-I00 and PID2021-127105NB-I00 from the Spanish Government. He is an Honorary Professor at Univ. Complutense de Madrid and member of IMI. 

\printbibliography

	\end{document}